\documentclass[reqno,a4paper,twoside]{amsart}

\usepackage[a4paper,margin=0.65in,footskip=0.25in]{geometry}

\usepackage{graphicx}
\usepackage{enumitem}

\setenumerate{label=\textnormal{(\arabic*)}}

\allowdisplaybreaks[4]

\usepackage{amsmath,amssymb,dsfont,verbatim,bm,mathrsfs,stmaryrd,mathabx}

\usepackage{mathtools}
\usepackage[raggedright]{titlesec}
\usepackage[utf8]{inputenc}
\usepackage[T1]{fontenc}

\usepackage[x11names,dvipsnames]{xcolor}
	\definecolor{egyptianblue}{rgb}{0.06, 0.2, 0.65}
	\definecolor{green(ncs)}{rgb}{0.0, 0.62, 0.42}

\usepackage{hyperref}
\hypersetup{
	colorlinks=true,
	citecolor=green(ncs),
	linkcolor= egyptianblue,
	filecolor=magenta,
	urlcolor=cyan,
	unicode=true,
	pdfpagemode=FullScreen,
}

\titleformat{\chapter}[display]
{\normalfont\huge\bfseries}{\chaptertitlename\\thechapter}{20pt}{\Huge}
\titleformat{\section}
{\normalfont\Large\bfseries\center}{\thesection}{1em}{}
\titleformat{\subsection}
{\normalfont\large\bfseries}{\thesubsection}{1em}{}
\titleformat{\subsubsection}
{\normalfont\normalsize\bfseries}{\thesubsubsection}{1em}{}
\titleformat{\paragraph}[runin]
{\normalfont\normalsize\bfseries}{\theparagraph}{1em}{}
\titleformat{\subparagraph}[runin]
{\normalfont\normalsize\bfseries}{\thesubparagraph}{1em}{}

\titlespacing*{\chapter} {0pt}{50pt}{40pt}
\titlespacing*{\section} {0pt}{3.5ex plus 1ex minus .2ex}{2.3ex plus .2ex}
\titlespacing*{\subsection} {0pt}{3.25ex plus 1ex minus .2ex}{1.5ex plus .2ex}
\titlespacing*{\subsubsection}{0pt}{3.25ex plus 1ex minus .2ex}{1.5ex plus .2ex}
\titlespacing*{\paragraph} {0pt}{3.25ex plus 1ex minus .2ex}{1em}
\titlespacing*{\subparagraph} {\parindent}{3.25ex plus 1ex minus .2ex}{1em}

\usepackage{float, subfig}
\usepackage{amsrefs}

\usepackage{titletoc}

\contentsmargin{2.55em}
\dottedcontents{section}[1.5em]{}{2em}{1pc}
\dottedcontents{subsection}[4.35em]{}{2.8em}{1pc}
\dottedcontents{subsubsection}[7.6em]{}{3.2em}{1pc}
\dottedcontents{paragraph}[10.3em]{}{3.2em}{1pc}

\usepackage[columns=3,rule=0pt]{idxlayout}
\setindexprenote{For convenience of the reader we list below almost all symbols used in this work, indicating the page in which each
of them is introduced. We only except those that are extremely standard.}

\makeindex

\newtheorem{theorem}{Theorem}[section]
\newtheorem{lemma}[theorem]{Lemma}
\newtheorem{proposition}[theorem]{Proposition}

\theoremstyle{definition}

\newtheorem{notation}[theorem]{Notation}

\theoremstyle{remark}
\newtheorem{remark}[theorem]{Remark}

\usepackage{tikz}
\usetikzlibrary{matrix,shapes}
\usetikzlibrary{cd}
\tikzcdset{diagrams={nodes={inner sep=2pt}}}
\usetikzlibrary{arrows}
\tikzcdset{arrow style=tikz, diagrams={>={Latex[round,length=3pt,width=2pt,inset=1.75pt]}}}

\DeclareMathOperator{\Ide}{Id}
\DeclareMathOperator{\Ext}{Ext}

\DeclareMathOperator{\Soc}{Soc}

\DeclareMathOperator{\Z}{Z}

\DeclareMathOperator{\ima}{Im}

\DeclareMathOperator{\Ho}{H}

\DeclareMathOperator{\ho}{H}

\DeclareMathOperator{\soc}{Soc}

\newcommand{\ov}{\overline}

\newcommand{\wh}{\widehat}
\newcommand{\wt}{\widetilde}

\newcommand{\xcirc}{}

\numberwithin{equation}{section}

\DeclareMathAlphabet{\mathpzc}{OT1}{pzc}{m}{it}

\usepackage{mathtools}
\newtagform{brackets}{[}{]}
\usetagform{brackets}

\let\origmaketitle\maketitle
\def\maketitle{
	\begingroup
	\def\uppercasenonmath##1{} % this disables uppercasing title
	\let\MakeUppercase\relax % this disables uppercasing authors
	\origmaketitle
	\endgroup
}

\usepackage{fancyhdr}
\usepackage[foot]{amsaddr}

\usepackage{adjustbox}

\begin{document}
	
\title{\large Cohomology of Trivial Linear Cycle Sets}

\author[Jorge A. Guccione]{Jorge A. Guccione$^{1}$}
\address{$^{1}$ Pontificia Universidad Cat\'olica del Per\'u, Secci\'on Matem\'aticas, PUCP, Avenida Universitaria~1801, San Miguel, Lima~32, Per\'u.}
\email{vander@dm.uba.ar}
	
\author[Jorge A. Guccione]{Juan J. Guccione$^{1}$}
%\address{Pontificia Universidad Cat\'olica del Per\'u, Secci\'on Matem\'aticas, PUCP, Av. Universitaria 1801, San Miguel, Lima 32, Per\'u.}
\email{jjguccione@gmail.com}
	
\author[Christian Valqui]{Christian Valqui$^{1}$}
%\address{Pontificia Universidad Cat\'olica del Per\'u, Secci\'on Matem\'aticas, PUCP, Av. Universitaria 1801, San Miguel, Lima 32, Per\'u.}
\email{cvalqui@pucp.edu.pe}

\thanks{This study was funded by the Vicerrectorado de Investigaci\'on (VRI) at the Pontificia Universidad Cat\'olica del Per\'u through grant CAP 2025 - PI1273}

	
\begin{abstract} we provide a complete classification of extensions of a trivial linear cycle set $H$ by an abelian group~$I$, under the assumption that both $H$ and $I$ are finite cyclic $p$-groups with $p$ odd, or $p=2$ and $|H| \le 4$. This yields an explicit parametrization of all possible extensions, offering a classification that is both comprehensive and computable. We also compute the socle and the center of all the linear cycle sets obtained.
\end{abstract}

%\begin{nouppercase}	
\maketitle
%\end{nouppercase}
	
\tableofcontents
	
\section*{Introduction}	

Motivated by the relevance of the Yang–Baxter equation (or equivalently the braid equation), Drinfeld proposed in~\cite{Dr} the study of set-theoretical solutions, which are closely related to a wide variety of algebraic structures such as affine torsors, solvable groups, Bieberbach groups, groups of I-type, Artin–Schelter regular algebras, Garside structures, biracks, Hopf algebras and left symmetric algebras, among others. See for example~\cites{BCJO, CJO, De1, ESS, GI2, GI4, GIM, R1, R2}. An important class of such solutions is given by the non-degenerate involutive ones, whose study led Rump to introduce linear cycle sets. This notion is equivalent to that of braces and bijective 1-cocycles, and has become a central tool in the structural study of set-theoretic solutions of the Yang–Baxter equation.

A fundamental problem in the theory of linear cycle sets is the description and classification of extensions (see~\cites{B, GGV1, GGV2, LV}. Given a linear cycle set $H$ and an abelian group $I$ endowed with the trivial structure, extensions of $H$ by $I$ can be described in terms of suitable actions $\blackdiamond\colon H \times I \to I$ and $\Yleft\colon I \times H \to I$, together with cohomological data arising from a filtered cochain complex. In this framework, it was proved that equivalence classes of extensions with fixed actions~$\blackdiamond$ and~$\Yleft$, correspond bijectively to the second cohomology group $H^2_{\blackdiamond,\Yleft}(H,I)$, reducing the classification problem to the determination of admissible actions and the computation of this group.

In previous works, several families of extensions have been studied, including the case where the adjoint group is finite abelian and particular examples where explicit computations can be carried out. A particularly relevant situation arises when the ideal $I$ is contained in the socle of the extension, which corresponds to the case $\Yleft = 0$. These extensions, although simpler from a technical point of view, already capture a rich structure and include many examples of interest.

The aim of this paper is to give a complete classification of extensions of linear cycle sets assuming that $H$ and~$I$ are finite cyclic $p$-groups, with $p$ odd or $p=2$ and $|H|\leq 4$. More precisely, we determine all possible actions~$\blackdiamond$ and~$\Yleft$, and explicitly compute the corresponding second cohomology groups, obtaining a full description of the extension classes. This allows us to parameterize all extensions in terms of explicit algebraic data, leading to a classification result that is both complete and computable. We also determine both the socle and the center of all the linear cycle sets obtained.

The paper is organized as follows. In Section~1 we recall the necessary background on linear cycle sets, extensions of linear cycle sets, and cohomology. In Section~2 we determine all the admissible actions $\blackdiamond$ and $\Yleft$ in the cases under consideration. Then, we compute the associated cohomology groups and establish the classification theorems. Likewise, we compute the socle and the center of the constructed linear cycle sets.

We work mainly in the ring $\mathbb{Z}_{p^r}$, of integers modulo $p^r$. Sometimes we use expressions such as $0 \le z < p^{\eta}$. This means that $z$ is the class in $\mathbb{Z}_{p^r}$ of an integer (also denoted $z$) that is greater than or equal to zero and less than $p^{\eta}$. For example, in Remarks~\ref{remark 1.11} and~\ref{remark 1.11'}, we can assume that
$$
0 < k' < p^{r-\beta}, \quad 0 < k'' \le p^{r-q} \quad \text{and} \quad 0 < k''' \le p^{r-q'}.
$$

\paragraph{Precedence of operations} The operations precedence in this paper is the following: the operators with the highest precedence are the unary operations $a\mapsto a^{\times n}$ (where $n\in \mathbb{Z}$) and the binary operation $(a,b)\mapsto {}^ab$; then come the operations $\cdot$ and $\blackdiamond$, that have equal precedence; then the multiplication $(a,b)\mapsto a\times b$; then the operator~$\Yleft$, and finally, the sum. Of course, as usual, this order of precedence can be modified by the use of parenthesis.

\section[Preliminaries]{Preliminaries}\label{Preliminaries}

A \emph{linear cycle set} is an additive abelian group $H$ equipped with a binary operation $\cdot$, such that all left translations
\[
h \mapsto l \cdot h
\]
are bijective, and the following conditions are satisfied:
\begin{equation}\label{compatibilidades cdot suma y suma cdot}
h\cdot (l+m)=h\cdot l + h\cdot m\quad\text{and}\quad (h+l)\cdot m = (h\cdot l)\cdot (h\cdot m)\quad\text{for all $h,l,m\in H$. }
\end{equation}

\subsection[The socle and the center of a linear cycle set]{The socle and the center of a linear cycle set}\label{The socle and the center of a linear cycle set}

Let $H$ be a linear cycle set. An {\em ideal} of $H$ is an additive subgroup $I$ of $H$, such that $h\cdot y \in I$ and $y\cdot h - h\in I$, for all~$h\in H$ and $y\in I$. An ideal $I$ of $H$ is called a {\em central ideal} if $y\cdot h=h$ and $h\cdot y=y$, for all $y\in I$ and~$h\in H$. The {\em socle} of $H$ is the ideal $\Soc(H)$ of all $y\in H$ such that $y\cdot h = h$, for all $h\in H$. The {\em center} of $H$, is the set $\Z(H)$ of all $y\in \Soc(H)$ such that $h\cdot y = y$, for all $h\in H$. The center of $H$ is a central ideal of $H$, and each central ideal of $H$ is included in $\Z(H)$.

\subsection[Extensions of linear cycle sets]{Extensions of linear cycle sets}\label{Extensions of linear cycle sets}

Let $I$ and $H$ be linear cycle sets. An {\em extension} $(\iota,B,\pi)$ of $H$ by $I$ is a short exact sequence
\begin{equation}
    \begin{tikzpicture}
    \begin{scope}[yshift=0cm,xshift=0cm, baseline]
			\matrix(BPcomplex) [matrix of math nodes, row sep=0em, text height=1.5ex, text
			depth=0.25ex, column sep=2.5em, inner sep=0pt, minimum height=5mm, minimum width =6mm]
			{0 & I & B & H & 0,\\};
			\draw[-{latex}] (BPcomplex-1-1) -- node[above=1pt,font=\scriptsize] {} (BPcomplex-1-2);
			\draw[-{latex}] (BPcomplex-1-2) -- node[above=1pt,font=\scriptsize] {$\iota$} (BPcomplex-1-3);
			\draw[-{latex}] (BPcomplex-1-3) -- node[above=1pt,font=\scriptsize] {$\pi$} (BPcomplex-1-4);
			\draw[-{latex}] (BPcomplex-1-4) -- node[above=1pt,font=\scriptsize] {} (BPcomplex-1-5);
    \end{scope}\label{sucesion exacta corta de grupos abelianos}
    \end{tikzpicture}
\end{equation}
of additive abelian groups, in which $B$ is a linear cycle set and both $\iota$ are $\pi$ are linear cycle sets morphisms. Two extensions $(\iota,B,\pi)$ and $(\iota',B',\pi')$ of $H$ by $I$ are {\em equivalent} if there exists a linear cycle set morphism $\phi\colon B\to B'$ such that $\pi'\xcirc \phi=\pi$ and $\phi\xcirc \iota=\iota'$. As in the case of extension of groups, necessarily $\phi$ is an isomorphism.

Assume that~\eqref{sucesion exacta corta de grupos abelianos} is an extension of linear cycle sets and let $s$ be a section of $\pi$ with $s(0)=0$. For the sake of simplicity, in the rest of this section we identify $I$ with $\iota(I)$ and we write $y$ instead of $\iota(y)$. For each~$b\in B$, there exist unique $y\in I$ and $h\in H$ such that $b=y+s(h)$. Let $\blackdiamond\colon H\times I\to I$ and $\Yleft \colon I\times H\to I$ be the maps defined by
$$
h\blackdiamond y \coloneqq s(h)\cdot y\quad\text{and}\quad y\Yleft h\coloneqq y\Yleft s(h) = y\cdot s(h)-s(h),
$$
respectively. From now on we assume that $I$ is trivial. By~\cite{GGV1}*{Proposition~3.3}, this implies that $\blackdiamond$ and $\Yleft$ do not depend on~$s$. Let $\beta\colon H\times H\to I$ and $f\colon H\times H\to I$ be the unique maps such that
\begin{align}
&y+s(h) + z+s(l)= y + z+ s(h) + s(l)= y+z + \beta(h,l) + s(h+l) \label{construccion de beta}\\
\shortintertext{and}
&(y+s(h))\cdot (z+s(l))= h\blackdiamond z + f(h,l) + h\blackdiamond y \Yleft h\cdot l + s(h\cdot l), \label{formula para cdot en I times H}
\end{align}
for $y,z\in I$ and $h,l\in H$.

\begin{notation}\label{notacion ext} Fix maps $\blackdiamond\colon H\times I\to I$ and $\Yleft \colon I\times H\to I$. We let $\Ext_{\blackdiamond,\Yleft}(H;I)$ denote the set of equivalence classes of  extensions $(\iota,B,\pi)$ of $H$ by $I$, such that $s(h)\cdot y = h\blackdiamond y$ and $y\cdot s(h)-s(h) = y\Yleft h$, for every section $s$ of $\pi$ with~$s(0) = 0$.
\end{notation}

\subsection[Building extensions of linear cycle sets]{Building extensions of linear cycle sets}\label{Building extensions of linear cycle sets}
Let $H$ be a linear cycle set and let $I$ be an additive abelian group. Given maps
$$
\blackdiamond\colon H\times I\to I,\quad \Yleft \colon I\times H\to I,\quad \beta\colon H\times H\to I\quad\text{and}\quad f\colon H\times H\to I,
$$
we let $I\times_{\beta,f}^{\blackdiamond,\Yleft} H$ denote $I\times H$, endowed with the binary operations $+$ and $\cdot$ defined by
\begin{align}
&(y+w_h) + (z+w_l) = y+z+\beta(h,l)+ w_{h+l}\label{formula para suma en I times H}\\
\shortintertext{and}
&(y+w_h)\cdot (z+w_l)= h\blackdiamond z + f(h,l) + h\blackdiamond y \Yleft h\cdot l + w_{h\cdot l},\label{formula para cdot en I times H'}
\end{align}
where we are writing $y$ instead of $(y,0)$ and $w_h$ instead of $(0,h)$.

\smallskip

Let $\triangleleft\colon I \times H\to I$ be the map defined by $y\triangleleft h\coloneqq  h \blackdiamond (y - y\Yleft h)$. For each $y\in I$ and $h\in H$, set $y^h\coloneqq y - y\Yleft h$. Consider $I$ endowed with the trivial linear cycle set structure and let $\iota\colon I\to I\times H$ and $\pi\colon I\times H\to H$ be the canonical maps. By the discussion at the beginning of~\cite{GGV1}*{Section~4} and~\cite{GGV1}*{Theorem~5.6 and Remark~5.12} we know that~$(\iota,I\times_{\beta,f}^{\blackdiamond,\Yleft} H, \pi)$ is an extension of linear cycle sets if and only if conditions~(1.12)--(1.19) in~\cite{GGV2}*{Subsection~1.3} are satisfied.

By~\cite{GGV1}*{Remark~5.14} two extensions $(\iota,I\times_{\beta,f}^{\blackdiamond,\Yleft} H,\pi)$ and $(\iota,I\times_{\beta',f'}^{\blackdiamond,\Yleft} H,\pi)$ of $H$ by $I$ are equivalent if and only if there exists a map $\varphi\colon H\to I$ satisfying:
\begin{align}
&\varphi(0) = 0,\qquad \varphi(h)-\varphi(h+l)+\varphi(l)=\beta(h,l) - \beta'(h,l)\label{equivalencia a nivel de qrupos}
\shortintertext{and}
&\varphi(h\cdot l) + f(h,l) = h\blackdiamond \varphi(l) + f'(h,l) + h\blackdiamond \varphi(h)\Yleft h\cdot l,\label{tercera condicion caso I trivial}
\end{align}
for all $h,l\in H$.

\smallskip

By~\cite{GGV1}*{Remark~4.9}, every extension $(\iota,B,\pi)$ of $H$ by $I$ is equivalent to an extension of the form $(\iota,I\times_{\beta,f}^{\blackdiamond,\Yleft} H,\pi)$.

\subsection[Cohomology of linear cycle sets]{Cohomology of linear cycle sets}

Let $H$ be a linear cycle set and let $I$ be an additive abelian group, considered as a trivial linear cycle set. Fix two maps $\blackdiamond\colon H\times I\to I$ and $\Yleft \colon I\times H\to I$. Assume that conditions~(1.12)--(1.14) in~\cite{GGV2}*{Subsection~1.3} are fulfilled.

In~\cite{GGV1}*{Section~7} we constructed a diagram $(\wh{C}_N^{**}(H,I),\partial_{\mathrm{h}},\partial_{\mathrm{v}},D)$. By applying to this diagram the same procedure used to form the total complex of a double complex, we obtain a cochain complex $(\wh{C}_N^*(H,I),\partial+D)$. As was proved in~\cite{GGV1}*{Proposition~7.9 and Corollary~7.10}, the second cohomology group $\Ho^2_{\blackdiamond,\Yleft}(H,I)$ of this complex is canonically isomorphic to $\Ext_{\blackdiamond,\Yleft}(H;I)$. This correspondence sends the cohomology class of a cocycle $(\beta,-f)\in \wh{C}_N^2(H,I)$ to the equivalence class of the extension $(\iota, I\times_{\beta,f}^{\blackdiamond,\Yleft} H,\pi)$.

Let $p$ be a prime number. In Section~2 we are going to classify the extension classes of a finite $p$-cyclic group, equipped with the trivial linear cycle set structure by another finite cyclic $p$-group $I$, considered as a trivial lineal cycle set. For this we will use the results obtained in the following subsection.

\subsection[The cohomology of a finite cyclic trivial cycle sets]{The cohomology of a finite cyclic trivial cycle sets}\label{sub 1.11}

In this subsection we specialize the results obtained in \cite{GGV2}*{Subsections~3.1 and~3.2} to the case $n=s=1$ (in other words, $H$ is a finite cyclic group of order $d$, equipped with the trivial cycle set structure), and examine certain~use\-ful properties for computing the extension classes of $H$ by a trivial cycle set $I$. Recall from~\cite{GGV2}*{Section~1.3} that, in order to build these extensions, we first must determine all the maps
$$
\blackdiamond\colon H\times I\to I\qquad\text{and}\qquad \Yleft\colon I\times H\to I
$$
satisfying~\cite{GGV2}*{conditions~(1.12)--(1.14)}. By~\cite{GGV2}*{Propo\-si\-tion~3.1}, having such a pair $(\blackdiamond,\Yleft)$ is equivalent to have~en\-domorphisms~$A$~and~$B$ of $I$ such that
\begin{equation}\label{condiciones para A y B}
A^d = \Ide,\quad (A-A\xcirc B)^d = \Ide\quad dB = 0, \quad\text{and}\quad B\xcirc A-A\xcirc B = B\xcirc A\xcirc B.
\end{equation}
The operators $\blackdiamond$ and $\Yleft$ are given by
\begin{equation}\label{formula para blackdiamond e Yleft}
h \blackdiamond y = A^hy\qquad\text{and}\qquad y\Yleft h = h B y.
\end{equation}
Furthermore, if
$$
T_1,T_2\colon I\oplus I\to I,\quad T_3\colon I\oplus I\to I^0 = 0\quad\text{and}\quad S\colon I\to I\oplus I,
$$
are as in \cite{GGV2}*{Remark~3.3}, then $e_1=a_1=1$, $T_3$ is the zero map,
\begin{equation}\label{T_1, T_2 y S}
T_1(\gamma,\mathfrak{f}_0) = d\mathfrak{f}_0+(\Ide-A)\gamma,\quad T_2(\gamma,\mathfrak{f}_0) = B\xcirc \sum_{j=1}^{d-1} j A^jf_0 + N(A) f_0 - B\gamma\quad\text{and}\quad  S(t)=(dt,B\xcirc At+At-t),
\end{equation}
where $N(A)\coloneqq \sum_{\ell=0}^{d-1}A^{\ell}$.

\smallskip

Let $\alpha_{\gamma}$ be as \cite{GGV2}*{Remark~2.7} and let $f_{\mathfrak{f}_0}$ be as the beginning of \cite{GGV2}*{Subsection~3.2}. For $0\le h,h'<d$, we have
\begin{equation}\label{alpha y f}
\alpha_{\gamma}(h,h') = \begin{cases} \gamma & \text{if $h+h'\ge d$,}\\ 0 & \text{if $h+h'<d$,}\end{cases}\qquad\text{and}\qquad f_{\mathfrak{f}_0}(h,h') = \sum_{j=0}^{h-1} h' A^{h-j-1} \mathfrak{f}_0 + \sum_{j=1}^{h-1} j h' B\xcirc A^j \mathfrak{f}_0,
\end{equation}
where for the computation of $f_{\mathfrak{f}_0}$, we use \cite{GGV2}*{Theorem~3.4}. Let
\begin{equation*}
T\colon \ker(T_1)\cap \ker(T_2)\to \wh{C}^{02}_N(H,I)\oplus \wh{C}^{11}_N(H,I)
\end{equation*}
be the linear map given by $T(\gamma,\mathfrak{f}_0) \coloneqq \bigl(\alpha_{\gamma},-f_{\mathfrak{f}_0}\bigr)$. By \cite{GGV2}*{Corollary~2.35}, we know that $T$ induces an isomorphism
\begin{equation}\label{formula T barra}
\ov{T}\colon \frac{\ker(T_1)\cap \ker(T_2)}{\ima(S)}\longrightarrow \Ho^2_{\blackdiamond,\Yleft}(H,I).
\end{equation}

\subsubsection[Some computations in two particular cases]{Some computations in two particular cases}\label{two particular cases}

\begin{enumerate}

\smallskip

\item[a)] Assume there exists $c\in \mathbb{Z}$ such that
\begin{equation}\label{paco1}
c(B\xcirc A+A-\Ide)=d\Ide.
\end{equation}
We define $\Phi\colon I^2\to I^2$ by $\Phi(x,y)\coloneqq (x+cy,y)$, and we set
\begin{equation}\label{wt T1, wt T2 y wt T3 caso a}
\wt{T}_1\coloneqq T_1\xcirc \Phi,\quad \wt{T}_2\coloneqq T_2\xcirc \Phi \quad\text{and}\quad \wt{S}\coloneqq \Phi^{-1}\xcirc S.
\end{equation}
Clearly $\Phi$ induces an isomorphism
\begin{equation}\label{Phi barra1'}
\ov{\Phi}\colon \frac{\ker \wt{T}_1\cap \ker \wt{T}_2}{\ima \wt{S}}\longrightarrow \frac{\ker T_1\cap \ker T_2}{\ima S}\simeq \Ho^2_{\blackdiamond,\Yleft}(H,I).
\end{equation}
A straightforward computation using the equalities~\eqref{T_1, T_2 y S} and~\eqref{paco1} yields
\begin{align}
& \wt{T}_1(x,y)=(\Ide-A)x+cB\xcirc Ay,\label{eq2}\\
&\wt{T}_2(x,y)=B\xcirc\sum_{j=1}^{d-1} j A^jy+N(A)y-cBy-Bx\label{eq3}\\
\shortintertext{and}
& \wt{S}(t)=(0,B\xcirc At+At-t).\label{eq4}
\end{align}
Moreover, from~\eqref{condiciones para A y B} and~\eqref{paco1}, we obtain that
\begin{equation}\label{pepito1}
c(B\xcirc A-B) = c(B\xcirc A\xcirc B+A\xcirc B-B) = dB = 0.
\end{equation}
Hence,
\begin{equation}\label{pepit1}
c(B+A-\Ide) = c(B\xcirc A+A-\Ide) = d\Ide,
\end{equation}
which implies
$$
cB+cA = (c+d)\Ide.
$$
Using this and the fact that $dB = 0$, we obtain
$$
cB^2 + cA\xcirc B = (c+d)B = cB \quad\text{and}\quad cB^2 + cB\xcirc A = (c+d)B = cB,
$$
which combined with~\eqref{pepito1}, gives
\begin{equation}\label{eq5}
cB\xcirc A = cA\xcirc B = cB\quad\text{and}\quad cB^2 = 0.
\end{equation}

\smallskip

\item[b)] Assume there exists $c\in \mathbb{Z}$ such that \begin{equation}\label{paco2}
B\xcirc A+A-\Ide = cd\Ide.
\end{equation}
We define $\Phi\colon I^2\to I^2$ by $\Phi(x,y)\coloneqq (x,y+cx)$, and we set
\begin{equation}\label{wt T1, wt T2 y wt T3 caso b}
\wt{T}_1\coloneqq T_1\xcirc \Phi,\quad \wt{T}_2\coloneqq T_2\xcirc \Phi\quad\text{and}\quad \wt{S}\coloneqq \Phi^{-1}\xcirc S.
\end{equation}
Clearly $\Phi$ induces an iso\-morphism
\begin{equation}\label{Phi barra'}
\ov{\Phi}\colon \frac{\ker \wt{T}_1\cap \ker \wt{T}_2}{\ima \wt{S}}\longrightarrow \frac{\ker T_1\cap \ker T_2}{\ima S}\simeq \Ho^2_{\blackdiamond,\Yleft}(H,I).
\end{equation}
A straightforward computation using the equalities~\eqref{T_1, T_2 y S} and~\eqref{paco2} yields
\begin{align}
& \wt{T}_1(x,y)=B\xcirc Ax+dy,\label{eqq2'}\\
&\wt{T}_2(x,y)=B\xcirc \sum_{j=1}^{d-1} j A^jy+N(A)y+cB\xcirc \sum_{j=1}^{d-1} j A^jx+cN(A)x-Bx\label{eqq3'}\\
\shortintertext{and}
& \wt{S}(t)=(dt,0).\label{eqq4'}
\end{align}
Moreover, from~\eqref{condiciones para A y B} and~\eqref{paco2}, we obtain that
\begin{equation}\label{pepito2}
B\xcirc A-B = B\xcirc A\xcirc B+A\xcirc B-B = cdB = 0.
\end{equation}
Hence,
\begin{equation*}
B+A-\Ide = B\xcirc A+A-\Ide = cd\Ide,
\end{equation*}
which implies
\begin{equation}\label{pepito4}
B+A = (cd+1)\Ide.
\end{equation}
Using this and the fact that $dB = 0$, we obtain
$$
B^2 + A\xcirc B = (cd+1)B = B \quad\text{and}\quad B^2 + B\xcirc A = (cd+1)B = B,
$$
which combined with~\eqref{pepito2}, gives
\begin{equation}\label{pepito3}
B\xcirc A = A\xcirc B = B\quad\text{and}\quad B^2 = 0.
\end{equation}
Using this and~\eqref{pepito4}, we obtain
$$
jB\xcirc A^j+A^j = jB\xcirc A^{j-1}+A^j = (B+A)^j = (cd+1)^j\Ide = \sum_{\ell=0}^j \binom{j}{\ell} c^{\ell}d^{\ell}\Ide,
$$
which implies that
$$
\qquad B\xcirc \sum_{j=1}^{d-1} j A^j+N(A) = \sum_{j=0}^{d-1} \sum_{\ell=0}^j \binom{j}{\ell} c^{\ell}d^{\ell}\Ide = \sum_{\ell=0}^{d-1} \sum_{j=\ell}^{d-1} \binom{j}{\ell} c^{\ell}d^{\ell}\Ide = \sum_{\ell=0}^{d-1} \binom{d}{\ell+1} c^{\ell}d^{\ell}\Ide.
$$
By this and~\eqref{pepito3}, the formulas for $\wt{T}_1$, $\wt{T}_2$ and $\wt{S}$ become
\begin{equation}\label{pepito5}
\wt{T}_1(x,y)=Bx+dy,\quad \wt{T}_2(x,y)=\sum_{\ell=0}^{d-1} \binom{d}{\ell+1} c^{\ell}d^{\ell}(y+cx) - Bx\quad\text{and}\quad \wt{S}(t) = (dt,0).
\end{equation}

\end{enumerate}

\subsection{A technical lemma}

Here, we establish a lemma needed to compute the operation $A$ when $H\coloneqq \mathbb{Z}_{p^{\eta}}$, endowed with the trivial cycle set structure, and $I\coloneqq \mathbb{Z}_{p^r}$.

\begin{lemma}\label{para ejemplos} Let $p\in \mathbb{N}$ be an odd prime and let $r\ge 1$. Let $0\le a<p^r$ and $\eta\in \mathbb{N}$. Then
\begin{equation*}
a^{p^{\eta}}\equiv 1\pmod{p^r} \Longleftrightarrow a^{p^{\eta_0}}\equiv 1\pmod{p^r} \Longleftrightarrow  a\equiv 1\pmod{p^{r-\eta_0}},
\end{equation*}
where $\eta_0\coloneqq \min(r-1,\eta)$.
\end{lemma}

\begin{proof} If $a\notin U(\mathbb{Z}_{p^r})$, none of the conditions are met. Else, by Euler Theorem, $a^{(p-1)p^{r-1}}\equiv 1\pmod{p^r}$. Consequently, $a^{p^{\eta}}\equiv 1\pmod{p^r}$ if and only if $a^{p^{\eta_0}}\equiv 1\pmod{p^r}$. When $r=1$, this finishes the proof. Hence, we can assume that $r>1$. Let~$s$ be a generator of $U(\mathbb{Z}_{p^r})$ and write $\ov{a} = s^l$, where $\ov{a}$ is the class of $a$ in $\mathbb{Z}_{p^r}$. Note that $a^{p^{\eta_0}} \equiv 1 \pmod{p^r}$ if and only if $(p-1)p^{r-1-\eta_0}\mid l$. Let $\hat{a}$ and $\hat{s}$ be the classes in $\mathbb{Z}_{p^{r-1}}$, of $\ov{a}$ and $s$, respectively. Since $\hat{s}$ is a generator of $U(\mathbb{Z}_{p^{r-1}})$ and $\hat{a} = \hat{s}^l$, the same argument as above proves that $(p-1)p^{(r-2)-(\eta_0-1)}= (p-1)p^{r-1-\eta_0}\mid l$ if and only if $a^{p^{\eta_0-1}}\equiv 1\pmod{p^{r-1}}$. An inductive argument concludes the proof.
\end{proof}

\begin{remark}\label{complemento para ejemplos} The previous lemma and its proof hold for $p=2$ when $1\le r\le 2$ (because then $U(\mathbb{Z}_{2^r})$ is cyclic).
\end{remark}

We also will use the following remark.

\begin{remark}\label{valuacion de p en combinatorio} Let $w,w'\in \mathbb{N}$ and write $w = p^tv$ and $w'=p^{t'}v'$ with $p\nmid v$ and $p\nmid v'$. Suppose that $t'\ge t$. Since
$$
\binom{p^{t'}v'}{p^tv} = \frac{p^{t'}v'}{p^tv} \binom{p^{t'}v'-1}{p^tv-1} = p^{t'-t} \frac{v'}{v} \binom{p^{t'}v'-1}{p^tv-1},
$$
we obtain that $p^{t'-t}\mid \binom{w'}{w}$.
\end{remark}

\section[The cohomology of trivial cyclic linear cycle sets of order \texorpdfstring{$p^{\eta}$}{pn}]{The cohomology of trivial cyclic linear cycle sets of order \texorpdfstring{$\pmb{p^{\eta}}$}{pn}}
Let $p,\eta,r\in \mathbb{N}$, with $p$ a prime number, and let $H\coloneqq \mathbb{Z}_{p^{\eta}}$, equipped with the trivial linear cycle set structure. In this section we explicitly compute all extension classes of $H$ by the finite cyclic $p$-group $I\coloneqq \mathbb{Z}_{p^r}$, considered as a trivial lineal cycle set. To achieve this, we first determine all operations $\blackdiamond$ and $\Yleft$ satisfying conditions~(1.12)--(1.14) in~\cite{GGV2}, which -according to the remarks at the beginning of Subsection~\ref{sub 1.11}- is equivalent to determining all endomor\-phisms~$A,B\colon I\to I$ satisfying the conditions specified in~\eqref{condiciones para A y B} with $d\coloneqq p^{\eta}$. Then, for each such pair~$(A,B)$,~we~com\-pute the corresponding cohomology group $\ho_{\blackdiamond,\Yleft}^2(H,I)$. To perform this final step, we use either~\eqref{Phi barra1'} or~\eqref{Phi barra'} to obtain a family of pairs $(\gamma,\mathfrak{f}_0)\in \ker T_1\cap \ker T_2$, where $T_1$ and $T_2$ are as in~\eqref{T_1, T_2 y S}. This family parameterizes a complete set of representatives of $2$-cocycles modulo coboundaries in $\wh{\mathcal{C}}_N^*(H,I)$, via the map
$$
T(\gamma,\mathfrak{f}_0) = (\alpha_{\gamma},-f_{\mathfrak{f}_0}).
$$
By the formulas~\eqref{formula para suma en I times H} and~\eqref{formula para cdot en I times H'} this is sufficient to explicitly construct the extension
$$
(\iota,I\times_{\alpha_{\gamma},-f_{\mathfrak{f}_0}}^{\blackdiamond,\Yleft} H,\pi).
$$
In each case, the parame\-ter family $(\gamma,\mathfrak{f}_0)$ corresponds bijectively to a set of pairs $(z_1,z_2)$, where $z_1$ and $z_2$ belong to subquotients of~$\mathbb{Z}_{p^r}$. Since~$I\coloneqq \mathbb{Z}_{p^r}$, there exist $a,b\in \mathbb{Z}_{p^r}$, such that the endomorphisms $A$ and $B$ are given mul\-tiplication by $a$ and~$b$, respectively. Therefore
\begin{equation}\label{pep2pr}
h\blackdiamond y = a^h y\qquad\text{and}\qquad y\Yleft h = b hy.
\end{equation}
Note that, since $ab=ba$, the conditions in~\eqref{condiciones para A y B} are satisfied if and only if
\begin{equation}\label{pep1}
a^{p^{\eta}} = 1,\quad p^{\eta} b = 0\quad\text{and}\quad b^2 = 0.
\end{equation}
By the simplicity of the formula for $\alpha_{\gamma}$ (see~\eqref{alpha y f}), we will focus on determining only $\gamma$, rather than computing~$\alpha_{\gamma}$. In many cases, the formula for $f_{\mathfrak{f}_0}$ can be simplified. Throughout all the computations we will assume that $0\le h< p^{\eta}$ for every $h\in H$.

\smallskip

%\begin{note} For the sake of simplicity in all cases we restrict ourselves to expressing the representative~$(\gamma,\mathfrak{f}_0)$ as a function of parameters~$z_1$ and $z_2$. For example, the phrase ``In this case $(\gamma,\mathfrak{f}_0)\in\{(p^{r-\eta}z_1,z_2):0\le z_1,z_2<p^{\eta}\}$'' should be interpreted as ``In this case a complete family $(\alpha_{\gamma},f_{\mathfrak{f}_0})$, of representatives of $2$-cocycles modulo coboundaries in $\wh{\mathcal{C}}_N^*(H,I)$, is parameterized by the pairs $(\gamma,\mathfrak{f}_0)\in \{(p^{r-\eta}z_1,z_2):0\le z_1,z_2<p^{\eta}\}$''. Subsections~\ref{subsection 2.1} and~\ref{subsection 2.1} must be readed taking into account this note.
%\end{note}

For the sake of brevity, from now on we write $E^{\blackdiamond,\Yleft}_{\gamma,\mathfrak{f}_0}$ instead of $I\times_{\alpha_{\gamma},-f_{\mathfrak{f}_0} }^{\blackdiamond,\Yleft} H$. Moreover, in the remainder of the paper, we will freely use the notations introduced in the following remarks.

\begin{remark}\label{remark 1.11} Assume for a moment that $b\ne 0$ and write $b = p^{\beta} k'$ with $0\le \beta<r$ and $k'\in \mathbb{Z}_{p^r}$ such that $p\nmid k'$. Clearly the second and third conditions in~\eqref{pep1} are satisfied if and only if $r\le\min(2\beta,\beta+\eta)$. When~$b=0$, we set~$\beta\coloneqq r$ and $k'\coloneqq 1$.
\end{remark}

\begin{remark}\label{remark 1.11'} We define $q$ and $q'$ as the largest integers such that
$$
q,q'\le r, \quad p^q\mid ba+a-1\text{ in }\mathbb{Z}_{p^r}\quad\text{and}\quad p^{q'}\mid ba-a+1\text{ in }\mathbb{Z}_{p^r}.
$$
Moreover we write
$$
ba+a-1 = p^qk''\quad\text{and}\quad ba-a+1 = p^{q'}k''',
$$
where $k'',k'''\in \mathbb{Z}_{p^r}^{\times}$. Note that if $a=1$, then $q=q'=\beta$ and $k''=k'''=k'$. Finally, we set $N(a)\coloneqq \sum_{i=0}^{p^{\eta}} a^i$.
\end{remark}

\subsection[Case \texorpdfstring{$H=\mathbb{Z}_{p^{\eta}}$}{H=Zpn} and \texorpdfstring{$I=\mathbb{Z}_{p^r}$}{I=Zpr} with \texorpdfstring{$r\le \eta$}{r<=n} if \texorpdfstring{$p$}{p} is odd, and \texorpdfstring{$r\le \min(2,\eta)$}{r<= min(2,n)} if \texorpdfstring{$p=2$}{p=2}]{Case \texorpdfstring{$\pmb{H=\mathbb{Z}_{p^{\eta}}}$}{H=Zpn} and \texorpdfstring{$\pmb{I=\mathbb{Z}_{p^r}}$}{I=Zpr} with \texorpdfstring{$\pmb{r\le \eta}$}{r<=n} if \texorpdfstring{$\pmb{p}$}{p} is odd, and \texorpdfstring{$\pmb{r\le \min(2,\eta)}$}{r<= min(2,n)} if \texorpdfstring{$\pmb{p=2}$}{p=2}}\label{subsection 2.1}

Let $H\coloneqq\left(\mathbb{Z}_{p^{\eta}},+\right)$, equipped with the trivial cycle set structure, let $I\coloneqq \mathbb{Z}_{p^r}$ and let $\blackdiamond\colon H\times I\to I$ and $\Yleft\colon I\times H\to I$ be two maps such that conditions~(1.12)--(1.14) in~\cite{GGV2} are satisfied. Assume $r\le \eta$ if $p$ is odd, and $r\le \min(2,\eta)$ if~$p=2$. Let $a,b\in \mathbb{Z}_{p^r}$ be as in~\eqref{pep2pr}. By Lemma~\ref{para ejemplos} and Remarks~\ref{complemento para ejemplos} and~\ref{remark 1.11}, we know that
$$
a = 1+pk\quad\text{and}\quad b=p^{\beta}k'\qquad\text{with $\beta\le r\le 2\beta$ and $p\nmid k'$.}
$$
If $a\ne 1$, then we write $k=p^us$ with $0\le u<r-1$ and $p\nmid s$. When $a=1$ we set $u\coloneqq r-1$ and $s\coloneqq 1$. Note that
\begin{equation}\label{accion1 ej 4'pr}
h\blackdiamond y = (1+pk)^h y = \sum_{i=0}^h \binom{h}{i} p^i k^i y.
\end{equation}
Let $(\gamma,\mathfrak{f}_0)\in \ker T_1\cap \ker T_2$ and let $f_{\mathfrak{f}_0}$ be as in~\eqref{alpha y f}. For $0\le h,h'<p^{\eta}$, we have
\begin{equation}\label{cociclo1 ej 3'}
\begin{aligned}
f_{\mathfrak{f}_0}(h,h') & = \sum_{j=0}^{h-1} (1+pk)^j \mathfrak{f}_0h'+ \sum_{j=0}^{h-1} j p^{\beta}k'(1+pk)^j \mathfrak{f}_0h'\\
& = \sum_{j=0}^{h-1} \sum_{\ell=0}^j\binom{j}{\ell} p^{\ell}k^{\ell}\mathfrak{f}_0h' + p^{\beta}k'\sum_{j=0}^{h-1} \sum_{\ell=0}^j j\binom{j}{\ell} p^{\ell}k^{\ell}\mathfrak{f}_0h'\\
& = \sum_{\ell=0}^{h-1} \sum_{j=\ell}^{h-1} \binom{j}{\ell} p^{\ell}k^{\ell}\mathfrak{f}_0h' + p^{\beta}k'\sum_{\ell=0}^{h-1} \sum_{j=\ell}^{h-1} j \binom{j}{\ell} p^{\ell}k^{\ell}\mathfrak{f}_0h'\\
& = \sum_{\ell=0}^{h-1} \binom{h}{\ell+1} p^{\ell}k^{\ell} \mathfrak{f}_0h'+p^{\beta}k'\sum_{\ell=0}^{h-1}\left((h-1) \binom{h}{\ell+1} - \binom{h}{\ell+2}\right) p^{\ell}k^{\ell}\mathfrak{f}_0h',
\end{aligned}
\end{equation}
where in the last equality we have use \cite{Spivey}*{identity~73}.

\begin{theorem}\label{prop 4.10} Assume that we are under the hypotheses established at the beginning of this subsection. A complete family $(\alpha_{\gamma},f_{\mathfrak{f}_0})$, of representatives of $2$-cocycles modulo coboundaries in $\wh{\mathcal{C}}_N^*(H,I)$, is parameterized by the pairs
$$
(\gamma,\mathfrak{f}_0)\in\{(p^{r-\min(u+1,\beta)}z_1,z_2)\in \mathbb{Z}_{p^r}\oplus \mathbb{Z}_{p^r}:0\le z_1<p^{\min(u+1,\beta)}\text{ and } 0\le z_2< p^q\},
$$
where $q$ is as in Remark~\ref{remark 1.11'}.
\end{theorem}

\begin{proof} Our aim is to determine a complete set of representatives of $\ker T_1\cap \ker T_2$~mod\-ule~$\ima S$. To begin with, note that, since $r\le \eta$, we have $p^{\eta}\mathfrak{f}_0 = 0$. Hence, from~\eqref{T_1, T_2 y S} we conclude that
$$
T_1(\gamma,\mathfrak{f}_0) = \gamma-a\gamma = - pk \gamma\quad\text{and}\quad S(t) = (0,p^qk''t).
$$
To compute $T_2$, arguing as in~\eqref{cociclo1 ej 3'}, we obtain that
$$
N(a)+b\sum_{j=1}^{p^{\eta}-1} j a^j = \sum_{\ell=0}^{p^{\eta}-1} \binom{p^{\eta}}{\ell+1} p^{\ell}k^{\ell} + p^{\beta}k'(p^{\eta}-1) \sum_{\ell=0}^{p^{\eta}-1} \binom{p^{\eta}}{\ell+1}p^{\ell}k^{\ell} - p^{\beta-1}k'\sum_{\ell=0}^{p^{\eta}-2} \binom{p^{\eta}}{\ell+2} p^{\ell+1}k^{\ell}.
$$
Write $\ell+1 = p^tv$ with $p\nmid v$ and $t\ge 0$. By Remark~\ref{valuacion de p en combinatorio}, we known that $p^{\eta-t}\mid \binom{p^{\eta}}{\ell+1}$. Therefore, the exponent of $p$ in~$\binom{p^{\eta}}{\ell+1} p^{\ell}$, is at least $\eta-t+\ell$. Since $\ell\ge t$ and $r\le \eta$, we have $\eta-t+\ell\ge \eta\ge r$, and hence $\binom{p^{\eta}}{\ell+1} p^{\ell}\equiv 0\pmod{p^r}$. Consequently,
$$
T_2(\gamma,\mathfrak{f}_0) = - B\gamma = - p^{\beta}k'\gamma.
$$
Thus,
\begin{equation}\label{para homologia 0}
\frac{\ker T_1\cap \ker T_2}{\ima S} =  p^{r-\min(u+1,\beta)}\mathbb{Z}_{p^r}\oplus\frac{\mathbb{Z}_{p^r}}{p^q\mathbb{Z}_{p^r}}.
\end{equation}
By this and~\eqref{formula T barra} the assertions in the statement are true.
\end{proof}

\begin{remark}\label{reducciones para f, etc si a=1} The formulas for $\blackdiamond$, $\Yleft$ and $f_{\mathfrak{f}_0}$ were presented in~\eqref{pep2pr}, \eqref{accion1 ej 4'pr} and~\eqref{cociclo1 ej 3'}. When $a=1$, we have:
$$
h\blackdiamond y = y\quad\text{and}\quad f_{\mathfrak{f}_0}(h,h')=\mathfrak{f}_0 hh'+ b\mathfrak{f}_0 \binom{h}{2}h'\qquad\text{for $0\le h,h'<p^{\eta}$.}
$$
\end{remark}

\begin{remark} In the proof of Theorem~\ref{prop 4.10}, we were able to determine the group structure of $\Ho^2_{\blackdiamond,\Yleft}(H,I)$. Specifically, by equality~\eqref{para homologia 0} we have
$$
\Ho^2_{\blackdiamond,\Yleft}(H,I)\simeq \mathbb{Z}_{p^{\min(u+1,\beta)}}\oplus \mathbb{Z}_{p^q}.
$$
\end{remark}

\subsubsection[The underlying additive group of the linear cycle set \texorpdfstring{$E^{\blackdiamond,\Yleft}_{\gamma,\mathfrak{f}_0}$}{E}]{The underlying additive group of the linear cycle set \texorpdfstring{$\pmb{E^{\blackdiamond,\Yleft}_{\gamma,\mathfrak{f}_0}}$}{E}}

We let $U\bigl(E^{\blackdiamond,\Yleft}_{\gamma,\mathfrak{f}_0}\bigr)$ denote the underlying additive group of the linear cycle set $E^{\blackdiamond,\Yleft}_{\gamma,\mathfrak{f}_0}$.

\begin{remark} Assume that we are under the assumptions of Theorem~\ref{prop 4.10} and write
$$
\ov{\beta}\coloneqq \min(\beta,u+1)\quad\text{and}\quad u_1\coloneqq \begin{cases} \text{the largest integer $k$ such that $p^k \mid z_1$} &\text{if $z_1\ne 0$,}\\ \ov{\beta} &\text{if $z_1 = 0$.}\end{cases}
$$
A direct computation shows that $p^{\eta+\ov{\beta}-u_1}$ is the order of $w_1\in U\bigl(E^{\blackdiamond,\Yleft}_{\gamma,\mathfrak{f}_0}\bigr)$. Since $\eta+\ov{\beta}-u_1\ge r$, we have
$$
p^{\eta+\ov{\beta}-u_1}(y+w_h) = p^{\eta+\ov{\beta}-u_1}y + p^{\eta+\ov{\beta}-u_1}w_h = 0\quad\text{for all $y\in I$ and $h\in H$.}
$$
Hence, the exponent of $U\bigl(E^{\blackdiamond,\Yleft}_{\gamma,\mathfrak{f}_0}\bigr)$ is $p^{\eta+\ov{\beta}-u_1}$, and consequently,
$$
U\bigl(E^{\blackdiamond,\Yleft}_{\gamma,\mathfrak{f}_0}\bigr) \simeq \mathbb{Z}_{p^{\eta+\ov{\beta}-u_1}}\oplus \mathbb{Z}_{p^{r-\ov{\beta}+u_1}}.
$$
\end{remark}

\subsubsection[The socle and the center]{The socle  and the center}

In this subsection, we work under the hypothesis of Theorem~\ref{prop 4.10}. Our objective is to compute the socle and the center of  $E^{\blackdiamond,\Yleft}_{\gamma,\mathfrak{f}_0}$ across the different cases that arise in that theorem.

\paragraph{Case $\pmb{p}$ odd}

\begin{proposition}\label{socalo caso 2 del primer teorema} Assume that $p$ is odd. For each pair of parameters $(\gamma,\mathfrak{f}_0)$ as specified in Theorem~\ref{prop 4.10}, we have
$$
\soc\bigl(E^{\blackdiamond,\Yleft}_{\gamma,\mathfrak{f}_0}\bigr) = \bigl\{y+w_h\in E^{\blackdiamond,\Yleft}_{\gamma,\mathfrak{f}_0}:p^{r-u-1}\mid h \text{ and } h\mathfrak{f}_0=-by\bigr\},
$$
where $u$ is at the beginning of Subsection~\ref{subsection 2.1} and the equality $h\mathfrak{f}_0=-by$ is performed in $\mathbb{Z}_{p^r}$.
\end{proposition}

\begin{proof} Let $y+w_h$ and $z+w_l$ in $E^{\blackdiamond,\Yleft}_{\gamma,\mathfrak{f}_0}$. Since by~\eqref{formula para cdot en I times H'} and~\eqref{pep2pr},
$$
(y+w_h)\cdot (z+w_l) = a^h z + f_{\mathfrak{f}_0}(h,l) + ba^hyl + w_l,
$$
we obtain that
$$
y+w_h\in \soc\bigl(E^{\blackdiamond,\Yleft}_{\gamma,\mathfrak{f}_0}\bigr) \quad\text{if and only if}\quad z = a^h z + f_{\mathfrak{f}_0}(h,l) + ba^h yl\text{ for all $l\in H$ and $z\in I$.}
$$
Taking $l=0$ and $z=1$, we obtain that if $y+w_h\in \soc\bigl(E^{\blackdiamond,\Yleft}_{\gamma,\mathfrak{f}_0}\bigr)$, then $a^h = 1$. Hence
$$
y+w_h\in \soc\bigl(E^{\blackdiamond,\Yleft}_{\gamma,\mathfrak{f}_0}\bigr) \quad\text{if and only if}\quad a^h = 1 \text{ and } f_{\mathfrak{f}_0}(h,l) = - byl\text{ for all $l\in H$.}
$$
Note that, since $pk = p^{u+1}s$ with $0\le u\le r-1$ and $p\nmid s$, from the lifting exponent lemma we obtain that
\begin{equation}\label{pepe9}
a^h = (1+p^{u+1}s)^h = 1\text{ if and only if $p^{r-u-1}\mid h$.}
\end{equation}
Furthermore, since $b = p^{\beta}k'$ and, by~\eqref{cociclo1 ej 3'},
$$
f_{\mathfrak{f}_0}(h,l) = \sum_{\ell=0}^{h-1} \binom{h}{\ell+1} p^{\ell}k^{\ell} \mathfrak{f}_0l + p^{\beta}k'\sum_{\ell=0}^{h-1}\left((h-1) \binom{h}{\ell+1} - \binom{h}{\ell+2}\right) p^{\ell}k^{\ell}\mathfrak{f}_0l,
$$
and element $y+w_h$ of $E^{\blackdiamond,\Yleft}_{\gamma,\mathfrak{f}_0}$ belongs to $\soc\bigl(E^{\blackdiamond,\Yleft}_{\gamma,\mathfrak{f}_0}\bigr)$ if and only if
\begin{equation}\label{pepe5}
p^{r-u-1}\mid h\quad\text{and}\quad \sum_{\ell=0}^{h-1} \binom{h}{\ell+1} p^{\ell}k^{\ell} \mathfrak{f}_0 + p^{\beta}k'\sum_{\ell=0}^{h-1}\left((h-1) \binom{h}{\ell+1} - \binom{h}{\ell+2}\right) p^{\ell}k^{\ell} \mathfrak{f}_0 = - p^{\beta}k'y.
\end{equation}
From the second condition we obtain that $p^{\beta}\mid \sum_{\ell=0}^{h-1} \binom{h}{\ell+1} p^{\ell}k^{\ell} \mathfrak{f}_0$. Hence $p^{2\beta}\mid p^{\beta} \sum_{\ell=0}^{h-1} \binom{h}{\ell+1} p^{\ell}k^{\ell} \mathfrak{f}_0$. Since $r\le 2\beta$, this implies that $p^{\beta} \sum_{\ell=0}^{h-1} \binom{h}{\ell+1} p^{\ell}k^{\ell} \mathfrak{f}_0 = 0$. Since moreover $pk = p^{u+1}s$, condition~\eqref{pepe5} becomes
\begin{equation}\label{pepe6}
p^{r-u-1}\mid h\quad\text{and}\quad \sum_{\ell=0}^{h-1} \binom{h}{\ell+1} p^{(u+1)\ell}s^{\ell} \mathfrak{f}_0 - p^{\beta}k'\sum_{\ell=0}^{h-1} \binom{h}{\ell+2} p^{(u+1)\ell}s^{\ell} \mathfrak{f}_0 = - p^{\beta}k'y.
\end{equation}
Assume that the first condition holds. We claim that the second one reduces to
\begin{equation}\label{pepe7}
h \mathfrak{f}_0 - p^{\beta}k'\binom{h}{2} \mathfrak{f}_0 = - p^{\beta}k'y.
\end{equation}
To check this, we first note that if $\beta\le u+1$ and $\ell\ge 1$, then $\beta+(u+1)\ell\ge 2\beta\ge r$, and so
$$
p^{\beta}k'\binom{h}{\ell+2} p^{(u+1)\ell} = 0;
$$
while if $\beta>u+1$, then
$$
p^{\beta}k'\binom{h}{\ell+2} p^{(u+1)\ell} = p^{\beta-u-1}k'\binom{h}{\ell+2} p^{(u+1)(\ell+1)}.
$$
Consequently, in order to obtain~\eqref{pepe7} it suffices to show that
\begin{equation}\label{pepe8}
\binom{h}{\ell+1} p^{(u+1)\ell} = 0\quad\text{for all $\ell\ge 1$.}
\end{equation}
We now proceed to prove this fact. Write $\ell+1 = p^tv$ with $p\nmid v$ and $t\ge 0$. Clearly, if $\ell\ge 1$, then $t\le \ell-1$.~There\-fore, if $t\ge r-u-1$, then
$$
(u+1)\ell \ge (u+1)(t+1) \ge (u+1)(r-u) = ru+r-(u+1)u\ge r,
$$
and so~\eqref{pepe8} holds. In the rest of the proof we assume that $t< r-u-1$. By the first condition in~\eqref{pepe6} and~Re\-mark~\ref{valuacion de p en combinatorio}, we known that $p^{r-u-1-t}\mid \binom{h}{\ell+1}$. Therefore, the exponent of $p$ in $\binom{h}{\ell+1} p^{(u+1)\ell}$, is at least
$$
r-u-1-t+(u+1)\ell\ge r-u+u\ell\ge r,
$$
and hence~\eqref{pepe8} holds. This finishes the proof of~\eqref{pepe7}. Therefore,
$$
y+w_h\in \soc\bigl(E^{\blackdiamond,\Yleft}_{\gamma,\mathfrak{f}_0}\bigr) \quad\text{if and only if}\quad p^{r-u-1}\mid h\text{ and } h \mathfrak{f}_0 - p^{\beta}k'\binom{h}{2} \mathfrak{f}_0 = - p^{\beta}k'y.
$$
But if $y+w_h$ satisfies the last condition, then $p^{\beta}\mid h\mathfrak{f}_0$ and hence $p^{2\beta}\mid p^{\beta}k' \binom{h}{2}\mathfrak{f}_0$. Since $2\beta\ge r$, the last equation~be\-comes $h\mathfrak{f}_0 = - by$. Consequently the proposition is true.
\end{proof}

\begin{remark} Assume that we are under the hypothesis of Proposition~\ref{socalo caso 2 del primer teorema}. If $\mathfrak{f}_0 = 0$, then
$$
\Soc\bigl(E^{\blackdiamond,\Yleft}_{\gamma,\mathfrak{f}_0}\bigr) = \bigl\{y+w_h\in E^{\blackdiamond,\Yleft}_{\gamma,\mathfrak{f}_0}: h\in p^{r-u-1}\mathbb{Z}_{p^{\eta}} \text{ and } y\in  p^{r-\beta} \mathbb{Z}_{p^r} \bigr\}.
$$
Consequently $\bigl|\Soc\bigl(E^{\blackdiamond,\Yleft}_{\gamma,\mathfrak{f}_0}\bigr)\bigr| = p^{\eta-r+u+1+\beta}$. Assume then that $\mathfrak{f}_0\ne 0$ and write $\mathfrak{f}_0 = z_2 = p^{u_2}s_2$ with~$p\nmid s_2$. Let $\bar{k}'$ be the inverse of $k'$ in $\mathbb{Z}_{p^r}$ and let $\xi\coloneqq \max(r-u-1,\beta-u_2)$. A direct computation shows that
$$
\Soc\bigl(E^{\blackdiamond,\Yleft}_{\gamma,\mathfrak{f}_0}\bigr) = \bigl\{y+w_h\in E^{\blackdiamond,\Yleft}_{\gamma,\mathfrak{f}_0}: h\in p^{\xi}\tilde{h} \text{ with } \tilde{h}\in \mathbb{Z}_{p^{\eta}}\text{ and } y\in -p^{\xi+u_2-\beta}\tilde{h}s_2\bar{k}' + p^{r-\beta} \mathbb{Z}_{p^r} \bigr\},
$$
which implies that $\bigl|\Soc\bigl(E^{\blackdiamond,\Yleft}_{\gamma,\mathfrak{f}_0}\bigr)\bigr| = p^{\eta-\xi+\beta}$.
\end{remark}

\begin{proposition}\label{centro caso 2 del primer teorema} Assume that $p$ is odd. For each pair of parameters $(\gamma,\mathfrak{f}_0)$ as specified in Theorem~\ref{prop 4.10}, we have
$$
\Z\bigl(E^{\blackdiamond,\Yleft}_{\gamma,\mathfrak{f}_0}\bigr) = \bigl\{y+w_h\in E^{\blackdiamond,\Yleft}_{\gamma,\mathfrak{f}_0}: p^{r-u-1}\mid h,\text{ } (a-1)y = by \text{ and } h\mathfrak{f}_0=-by \bigr\}.
$$
where $u$ is at the beginning of Subsection~\ref{subsection 2.1} and the equalities are performed in $\mathbb{Z}_{p^r}$.
\end{proposition}

\begin{proof} Let $y+w_h\in \soc\bigl(E^{\blackdiamond,\Yleft}_{\gamma,\mathfrak{f}_0}\bigr)$. Arguing as in the proof of Proposition~\ref{socalo caso 2 del primer teorema}, we obtain
$$
y+w_h\in \Z\bigl(E^{\blackdiamond,\Yleft}_{\gamma,\mathfrak{f}_0}\bigr) \quad\text{if and only if}\quad y = a^l y + f_{\mathfrak{f}_0}(l,h) + ba^l yh\text{ for all $l\in H$.}
$$
Taking $l=0$, we obtain that
$$
y+w_h\in \Z\bigl(E^{\blackdiamond,\Yleft}_{\gamma,\mathfrak{f}_0}\bigr) \quad\text{if and only if}\quad byh = 0 \text{ and } y = a^l y + f_{\mathfrak{f}_0}(l,h)\text{ for all $l\in H$.}
$$
Note now that since $p^{r-u-1}\mid h$ and $pk = p^{u+1}s$, from~\eqref{cociclo1 ej 3'} we obtain
$$
f_{\mathfrak{f}_0}(l,h) = \mathfrak{f}_0 lh + p^{\beta}k'\binom{l}{2}h\mathfrak{f}_0 = \mathfrak{f}_0 lh = -byl\qquad\text{for $0\le h,l<p^{\eta}$,}
$$
where the second equality holds because $p^{\beta}k'h\mathfrak{f}_0 = - p^{2\beta} (k')^2$ and $r\le 2\beta$. Hence
$$
y = a^l y + f_{\mathfrak{f}_0}(l,h) \quad\text{if and only if}\quad (a^l-1) y = bl y
$$
We claim that the equality $(a-1)y = by$, corresponding to $l=1$, entails all the other ones. Indeed, since $b^2 = 0$ and $a-1$ divides $1+a+\cdots+a^{l-1}-l$, there exists $x\in \mathds{Z}_{p^r}$ such that
$$
(a^l-1-bl) y = (a-1)(1+a+\cdots+a^{l-1}-l)y = (a-1)^2xy = b^2 xy = 0,
$$
as desired. Moreover, since $p^{u+1}\mid a-1$ and $p^{r-u-1}\mid h$, from $(a-1)y = by$, we also obtain that $bhy = h(a-1)y = 0$. The proposition follows immediately from these facts.
\end{proof}

\begin{remark} Let $\delta$ be the largest integer $k\le r$ such that $p^k\mid a-1-b = p^{u+1}s-p^{\beta}k'$. Observe that, if~$u+1\ne \beta$, then $\delta = \min(u+1,\beta)$, while if $u+1 = \beta$, then $\delta\ge u+1$. Clearly
$$
\Z\bigl(E^{\blackdiamond,\Yleft}_{\gamma,\mathfrak{f}_0}\bigr) = \bigl\{y+w_h\in E^{\blackdiamond,\Yleft}_{\gamma,\mathfrak{f}_0}: p^{r-\delta}\mid y,\text{ } p^{r-u-1}\mid h \text{ and } h\mathfrak{f}_0=-by \bigr\}.
$$
Since $\min(\delta,\beta) = \min(u+1,\beta)$, if $\mathfrak{f}_0 = 0$, then
$$
\Z\bigl(E^{\blackdiamond,\Yleft}_{\gamma,\mathfrak{f}_0}\bigr) = \bigl\{y+w_h\in E^{\blackdiamond,\Yleft}_{\gamma,\mathfrak{f}_0}: y\in p^{r-\min(u+1,\beta)}\mathbb{Z}_{p^r} \text{ and }  h\in p^{r-u-1}\mathbb{Z}_{p^{\eta}} \bigr\}.
$$
Consequently $\bigl|\Z\bigl(E^{\blackdiamond,\Yleft}_{\gamma,\mathfrak{f}_0}\bigr)\bigr| = p^{\eta-r+u+1+\min(u+1,\beta)}$. Assume then that $\mathfrak{f}_0\ne 0$ and write $\mathfrak{f}_0 = z_2 = p^{u_2}s_2$ with~$p\nmid s_2$. Let~$\bar{k}'$ be the inverse of $k'$ in $\mathbb{Z}_{p^r}$ and let $\xi\coloneqq \max(r-u-1,\beta-u_2)$. A direct computation shows that
$$
\Z\bigl(E^{\blackdiamond,\Yleft}_{\gamma,\mathfrak{f}_0}\bigr) = \bigl\{y+w_h\in E^{\blackdiamond,\Yleft}_{\gamma,\mathfrak{f}_0}: h\in p^{\xi}\tilde{h} \text{ with } \tilde{h}\in \mathbb{Z}_{p^{\eta}},\text{ } y\in p^{r-\delta}\mathbb{Z}_{p^r}\text{ and } y\in -p^{\xi+u_2-\beta}\tilde{h}s_2\bar{k}' + p^{r-\beta} \mathbb{Z}_{p^r} \bigr\}.
$$
Assume that $\delta = \min(u+1,\beta)$. Since $p^{r-\delta}\mid y$, in this case $hf_0 = - by=0$. Thus,
$$
\Z\bigl(E^{\blackdiamond,\Yleft}_{\gamma,\mathfrak{f}_0}\bigr) = \bigl\{y+w_h\in E^{\blackdiamond,\Yleft}_{\gamma,\mathfrak{f}_0}: h\in p^{\xi'}\mathbb{Z}_{p^{\eta}} \text{ and }  y\in p^{r-\delta}\mathbb{Z}_{p^r}\bigr\},
$$
where $\xi'\coloneqq \max(r-u-1,r-u_2)$, and hence $\bigl|\Z\bigl(E^{\blackdiamond,\Yleft}_{\gamma,\mathfrak{f}_0}\bigr)\bigr| = p^{\eta-\xi'+\delta}$. \begin{comment}
$$
y\in p^{r-\delta} \mathbb{Z}_{p^r}\subseteq p^{r-\beta} \mathbb{Z}_{p^r}
$$
and so $f_0 h = -b y = 0$. Hence $h\in p^{r-u_2} \mathbb{Z}_{p^r}$.
We have
$$
\Z\bigl(E^{\blackdiamond,\Yleft}_{\gamma,\mathfrak{f}_0}\bigr) = \bigl\{y+w_h\in E^{\blackdiamond,\Yleft}_{\gamma,\mathfrak{f}_0}: y\in p^{\xi}\mathbb{Z}_{p^r} \text{ and }  h\in p^{\xi'}\mathbb{Z}_{p^{\eta}} \bigr\},
$$
where $\xi'\coloneqq \max(r-u-1,r-\beta,r-u_2)$. Consequently $\bigl|\Z\bigl(E^{\blackdiamond,\Yleft}_{\gamma,\mathfrak{f}_0}\bigr)\bigr| = p^{r+\eta-\xi-\xi'}$.
$$
p^{u_2}s_2 h = - p^{\beta} k' y
$$
$$
(p^{u+1}s-p^{\beta}k')y=0\quad p^{u_2}s_2 p^{r-u-1}\tilde{h} = - p^{\beta}k'y = p^{u+1}sy
$$
Let $p$, $a$ and $(\gamma,\mathfrak{f}_0)$ be as in Proposition~\ref{centro caso 2 del primer teorema} and let $\xi\coloneqq \max(r-u-1,r-\beta)$. If $\mathfrak{f}_0 = 0$, then
$$
\Z\bigl(E^{\blackdiamond,\Yleft}_{\gamma,\mathfrak{f}_0}\bigr) = \bigl\{y+w_h\in E^{\blackdiamond,\Yleft}_{\gamma,\mathfrak{f}_0}: y\in p^{\xi}\mathbb{Z}_{p^r} \text{ and }  h\in p^{r-u-1}\mathbb{Z}_{p^{\eta}} \bigr\}.
$$
Consequently $\bigl|\Z\bigl(E^{\blackdiamond,\Yleft}_{\gamma,\mathfrak{f}_0}\bigr)\bigr| = p^{\eta+u-\xi+1}$.
\end{comment}
\end{remark}

\paragraph{Case $\pmb{p=2}$}

\begin{remark}\label{caso 1 p=2} Assume that $p=2$ and let $(\gamma,\mathfrak{f}_0)$ be as in Theorem~\ref{prop 4.10}. Let $r=1$, which implies that $a=1$ and~$b=0$. Since by~\eqref{formula para cdot en I times H'},~\eqref{pep2pr} and Remark~\ref{reducciones para f, etc si a=1}
$$
(y+w_h)\cdot (z+w_l) = z + \mathfrak{f}_0hl + w_l,
$$
we obtain that
$$
y+w_h\in \soc\bigl(E^{\blackdiamond,\Yleft}_{\gamma,\mathfrak{f}_0}\bigr)\quad\text{if and only if}\quad \mathfrak{f}_0hl = 0 \text{ for all $l\in H$.}
$$
Thus,
$$
\soc\bigl(E^{\blackdiamond,\Yleft}_{\gamma,\mathfrak{f}_0}\bigr) = \begin{cases} E^{\blackdiamond,\Yleft}_{\gamma,\mathfrak{f}_0} & \text{if $\mathfrak{f}_0=0$,}\\ \{y+w_h\in E^{\blackdiamond,\Yleft}_{\gamma,\mathfrak{f}_0}:2\mid h\} & \text{if $\mathfrak{f}_0=1$.}\end{cases}
$$
A direct computation shows that
$$
\Z\bigl(E^{\blackdiamond,\Yleft}_{\gamma,\mathfrak{f}_0}\bigr) = \soc\bigl(E^{\blackdiamond,\Yleft}_{\gamma,\mathfrak{f}_0}\bigr).
$$
Assume now that $r=2$. Then $(a,k)\in\{(1,0),(3,1)\}$ and $b\in \{0,2\}$. Since by~\eqref{formula para cdot en I times H'} and~\eqref{pep2pr}
$$
(y+w_h)\cdot (z+w_l) = a^hz + f_{\mathfrak{f}_0}(h,l) + ba^hyl + w_l,
$$
we obtain that
$$
y+w_h\in \soc\bigl(E^{\blackdiamond,\Yleft}_{\gamma,\mathfrak{f}_0}\bigr) \quad\text{if and only if}\quad z = a^h z + f_{\mathfrak{f}_0}(h,l) + ba^h yl\text{ for all $l\in H$ and $z\in I$.}
$$
Taking $l=0$ and $z=1$, we obtain that if $y+w_h\in \soc\bigl(E^{\blackdiamond,\Yleft}_{\gamma,\mathfrak{f}_0}\bigr)$, then $a^h = 1$. Hence
\begin{equation}\label{pepe4}
y+w_h\in \soc\bigl(E^{\blackdiamond,\Yleft}_{\gamma,\mathfrak{f}_0}\bigr) \quad\text{if and only if}\quad a^h = 1 \text{ and } f_{\mathfrak{f}_0}(h,l) = - byl\text{ for all $l\in H$.}
\end{equation}
Note that
\begin{equation}\label{pepe4'}
a^h=1 \text{ if and only if } (a=1) \text{ or } (a=3\text{ and } 2\mid h).
\end{equation}
Moreover, by~\eqref{cociclo1 ej 3'},
\begin{equation}\label{pepe4''}
f_{\mathfrak{f}_0}(h,l) = \begin{cases} \mathfrak{f}_0 hl+ b\mathfrak{f}_0 \binom{h}{2}l & \text{if $a=1$,}\\ \mathfrak{f}_0 hl + 2\mathfrak{f}_0 \binom{h}{2}l + b\mathfrak{f}_0 \binom{h}{2} l &\text{if $a=3$,}\end{cases} = \begin{cases} \mathfrak{f}_0hl & \text{if $a=1$ and $b= 0$,}\\ \mathfrak{f}_0h^2l & \text{if $a=1$ and $b=2$,}\\ \mathfrak{f}_0h^2l & \text{if $a=3$ and $b=0$,}\\ \mathfrak{f}_0hl &\text{if $a=3$ and $b=2$.}\end{cases}
\end{equation}
Furthermore, by Theorem~\ref{prop 4.10},
\begin{equation}\label{2pepe4''}
\mathfrak{f}_0 \in \begin{cases} \{0,1,2,3\} & \text{if ($a=1$ and $b=0$) or ($a=3$ and $b=2$),}\\ \{0,1\} & \text{if ($a=1$ and $b=2$) or ($a=3$ and $b=0$).} \end{cases}
\end{equation}
Combining this with~\eqref{pepe4}, \eqref{pepe4'} and~\eqref{pepe4''}, we obtain that
\begin{equation}\label{3pepe4''}
\soc\bigl(E^{\blackdiamond,\Yleft}_{\gamma,\mathfrak{f}_0}\bigr) = \begin{cases} E^{\blackdiamond,\Yleft}_{\gamma,\mathfrak{f}_0} & \text{if $\mathfrak{f}_0=0$, $a=1$ and $b=0$,}\\ \{y+w_h\in E^{\blackdiamond,\Yleft}_{\gamma,\mathfrak{f}_0}:4\mid h\} & \text{if $\mathfrak{f}_0=1$, $a=1$ and $b=0$,}\\ \{y+w_h\in E^{\blackdiamond,\Yleft}_{\gamma,\mathfrak{f}_0}:2\mid h\} & \text{if $(\mathfrak{f}_0,a)\in\{(0,3),(1,3),(2,1),(3,1)\}$ and $b=0$,}\\ \{y+w_h\in E^{\blackdiamond,\Yleft}_{\gamma,\mathfrak{f}_0}:2\mid y\} & \text{if $(\mathfrak{f}_0,a)\in\{(0,1),(2,3)\}$ and $b=2$,}\\ \{y+w_h\in E^{\blackdiamond,\Yleft}_{\gamma,\mathfrak{f}_0}:2\mid y\text{ and } 2\mid h\} & \text{if $(\mathfrak{f}_0,a)\in\{(0,3),(1,1)\}$ and $b=2$,}\\ \{y+w_h\in E^{\blackdiamond,\Yleft}_{\gamma,\mathfrak{f}_0}: (y,h)\in \mathfrak{A}\} & \text{if $\mathfrak{f}_0\in\{1,3\}$, $a=3$ and $b=2$.}\end{cases}
\end{equation}
where $\mathfrak{A}\coloneqq \{(y,h)\in \mathbb{Z}_4\oplus \mathbb{Z}_{p^{\eta}}:(y\in\{0,2\} \text{ and } 4\mid h) \text{ or } (y\in\{1,3\} \text{ and } h\equiv 2\pmod{4})\}$. Arguing as in the proof of Proposition~\ref{centro caso 2 del primer teorema}, we can see that an element $y+w_h\in \soc\bigl(E^{\blackdiamond,\Yleft}_{\gamma,\mathfrak{f}_0}\bigr)$ belongs to $\Z\bigl(E^{\blackdiamond,\Yleft}_{\gamma,\mathfrak{f}_0}\bigr)$ if and only if $$
byh=0\qquad\text{and}\qquad y = a^l y + f_{\mathfrak{f}_0}(l,h)\quad\text{for $1\le l< 2^{\eta}$.}
$$
Using this and~\eqref{pepe4''}--\eqref{3pepe4''}, we obtain that
$$
\Z\bigl(E^{\blackdiamond,\Yleft}_{\gamma,\mathfrak{f}_0}\bigr) = \begin{cases} E^{\blackdiamond,\Yleft}_{\gamma,\mathfrak{f}_0} & \text{if $\mathfrak{f}_0=0$, $a=1$ and $b=0$,}\\ \{y+w_h\in E^{\blackdiamond,\Yleft}_{\gamma,\mathfrak{f}_0}:4\mid h\} & \text{if $\mathfrak{f}_0\in\{1,3\}$, $a=1$ and $b=0$,}\\ \{y+w_h\in E^{\blackdiamond,\Yleft}_{\gamma,\mathfrak{f}_0}:2\mid y \text{ and } 2\mid h\} & \text{if $(\mathfrak{f}_0,b)\in\{(0,0),(0,2),(2,2)\}$ and $a=3$,}\\
\{y+w_h\in E^{\blackdiamond,\Yleft}_{\gamma,\mathfrak{f}_0}: (y,h)\in \mathfrak{A}\} & \text{if $(\mathfrak{f}_0,b)\in\{(1,0),(1,2),(3,2)\}$ and $a=3$,}\\ \{y+w_h\in E^{\blackdiamond,\Yleft}_{\gamma,\mathfrak{f}_0}: 2\mid h\} & \text{if $\mathfrak{f}_0=2$, $a=1$ and $b=0$,}\\ \{y+w_h\in E^{\blackdiamond,\Yleft}_{\gamma,\mathfrak{f}_0}:2\mid y\} & \text{if $\mathfrak{f}_0=0$, $a=1$ and $b=2$,}\\
\{y+w_h\in E^{\blackdiamond,\Yleft}_{\gamma,\mathfrak{f}_0}:2\mid y \text{ and } 4\mid h\} & \text{if $\mathfrak{f}_0=1$, $a=1$ and $b=2$.}\end{cases}
$$
\end{remark}

\subsection[Case \texorpdfstring{$H=\mathbb{Z}_{p^{\eta}}$}{H=Zpn} and \texorpdfstring{$I=\mathbb{Z}_{p^r}$}{I=Zpr} with \texorpdfstring{$r>\eta$}{r>n} if \texorpdfstring{$p$}{p} is odd, and \texorpdfstring{$(\eta,r)=(1,2)$}{(n,r)=(1,2)} if \texorpdfstring{$p=2$}{p=2}]{Case \texorpdfstring{$\pmb{H=\mathbb{Z}_{p^{\eta}}}$}{H=Zpn} and \texorpdfstring{$\pmb{I=\mathbb{Z}_{p^r}}$}{I=Zpr} with \texorpdfstring{$\pmb{r>\eta}$}{r>n} if \texorpdfstring{$\pmb{p}$}{p} is odd, and \texorpdfstring{$\pmb{(\eta,r)=(1,2)}$}{(n,r)=(1,2)} if \texorpdfstring{$\pmb{p=2}$}{p=2}}\label{subseccion 2.2}

Let $H\coloneqq\left(\mathbb{Z}_{p^{\eta}},+\right)$, equipped with the trivial cycle set structure, let $I\coloneqq \mathbb{Z}_{p^r}$ and let $\blackdiamond\colon H\times I\to I$ and $\Yleft\colon I\times H\to I$ be two maps such that conditions~(1.12)--(1.14) in~\cite{GGV2} are satisfied. Assume $\eta<r$ if $p$ is odd, and $(\eta,r) = (1,2)$ if~$p=2$. Let $a,b\in \mathbb{Z}_{p^r}$ be as in~\eqref{pep2pr}. By Lemma~\ref{para ejemplos} and Remarks~\ref{complemento para ejemplos} and~\ref{remark 1.11}, we know that $$
a = 1+p^{r-\eta}k\quad\text{and}\quad b=p^{\beta}k'\qquad\text{with $\beta\le r\le \min(2\beta,\beta+\eta)$ and $p\nmid k'$.}
$$
If $a\ne 1$, then we write $k=p^us$ with $0\le u<\eta$ and $p\nmid s$. When $a=1$ we set $u\coloneqq \eta$ and $s\coloneqq 1$. Note that
\begin{equation}\label{accion1 ej 4'}
h\blackdiamond y = (1+p^{r-\eta}k)^h y = \sum_{i=0}^h \binom{h}{i} p^{(r-\eta)i} k^i y.
\end{equation}
Let $(\gamma,\mathfrak{f}_0)\in \ker T_1\cap \ker T_2$ and let $f_{\mathfrak{f}_0}$ be as in~\eqref{alpha y f}. For $0\le h,h'<p^{\eta}$, we have
\begin{equation}\label{cociclo1 ej 4'}
\begin{aligned}
f_{\mathfrak{f}_0}(h,h') & = \sum_{j=0}^{h-1} (1+p^{r-\eta}k)^j \mathfrak{f}_0h'+ \sum_{j=0}^{h-1} j p^{\beta}k'(1+p^{r-\eta}k)^j \mathfrak{f}_0h'\\
& = \sum_{j=0}^{h-1} \sum_{\ell=0}^j\binom{j}{\ell} p^{(r-\eta)\ell}k^{\ell}\mathfrak{f}_0h' + p^{\beta}k'\sum_{j=0}^{h-1} \sum_{\ell=0}^j j\binom{j}{\ell} p^{(r-\eta)\ell}k^{\ell}\mathfrak{f}_0h'\\
& = \sum_{\ell=0}^{h-1} \sum_{j=\ell}^{h-1} \binom{j}{\ell} p^{(r-\eta)\ell}k^{\ell}\mathfrak{f}_0h' + p^{\beta}k'\sum_{\ell=0}^{h-1} \sum_{j=\ell}^{h-1} j \binom{j}{\ell} p^{(r-\eta)\ell}k^{\ell}\mathfrak{f}_0h'\\
& = \sum_{\ell=0}^{h-1} \binom{h}{\ell+1} p^{(r-\eta)\ell}k^{\ell} \mathfrak{f}_0h'+p^{\beta}k'\sum_{\ell=0}^{h-1}\left((h-1) \binom{h}{\ell+1} - \binom{h}{\ell+2}\right) p^{(r-\eta)\ell}k^{\ell}\mathfrak{f}_0h'.
\end{aligned}
\end{equation}

\begin{remark}\label{caso eta le q} Let $q$, $q'$, $k''$ and $k'''$ be as defined in Remark~\ref{remark 1.11'}. Assume that $\eta\le q$, so that we are in the case~b) of Subsubsection~\ref{two particular cases}, with $d\coloneqq p^{\eta}$ and $c\coloneqq p^{q-\eta}k''$. More concretely,
\begin{equation}\label{pep0N}
p^{\eta} c = p^q k'' = ba+a-1 = b+a-1 = p^{\beta}k' + p^{r-\eta}k,
\end{equation}
where we used Remark~\ref{remark 1.11'} and the identity $b=ba$ from~\eqref{pepito2}. Similarly, we have
\begin{equation}\label{pepito3'N}
p^{q'} k''' = ba-a+1 = b-a+1 = p^{\beta}k' - p^{r-\eta}k = 2p^{\beta}k' - p^qk'',
\end{equation}
where the last equality holds by~\eqref{pep0N}. By Remark~\ref{remark 1.11}, if $a=1$, then
\begin{equation}\label{un caso'}
q=q'=\beta\ge r-\eta.
\end{equation}
Suppose now that $q \ne q'$, $q \ne \beta$, or $q'\ne \beta$, which implies $a\ne 1$. If $p$ is odd, then applying~\eqref{pep0N} and~\eqref{pepito3'N}, we deduce that
\begin{equation}\label{tres casos'}
q=q'=r-\eta+u<\beta,\quad q' =\beta = r-\eta+u < q\quad\text{or}\quad q = \beta = r-\eta+u < q'.
\end{equation}
Note also that, since $\eta<r$, from~\eqref{tres casos'} it follows that $u<\min(q,q',\beta)$.
\end{remark}

\begin{remark}\label{caso q < eta} Assume that $q<\eta$, let $k''$ be as defined in Remark~\ref{remark 1.11'} and let $\kappa\in \mathbb{Z}_{p^r}$ be such that $p^{\eta}\kappa k'' = p^{\eta}$ (thus, we can take~$1<\kappa < p^{r-\eta}$ such that $\kappa k''\equiv 1\pmod{p^{r-\eta}})$. Clearly, we are in the case~a) of Sub\-sub\-section~\ref{two particular cases}, with $d\coloneqq p^{\eta}$ and $c\coloneqq p^{\eta-q}\kappa$. More concretely
\begin{equation}\label{pepit2}
p^{\eta}= p^{\eta-q}\kappa p^qk'' = c(ba+a-1) = c(b+a-1),
\end{equation}
where the last equality holds by~\eqref{pepit1}. Since $\eta<r$, this implies that $q$ is the largest~inte\-ger, lesser or equal than $r$, such that
\begin{equation*}%\label{pepito4''}
p^q\mid b+a-1 = p^{\beta}k' + p^{r-\eta}k.
\end{equation*}
Clearly, if $a=1$, then $q=q'=\beta$. Assume that $q\ne q'$, $q\ne \beta$ or $q'\ne \beta$, which implies that $a\ne 1$. We claim that if~$p$ is odd, then
\begin{equation}\label{tres casos}
q=q'=r-\eta+u<\beta,\quad q' =\beta = r-\eta+u < q\quad\text{or}\quad q = \beta = r-\eta+u < q'.
\end{equation}
From the equalities
\begin{align*}
& p^qk'' = ba+a-1 = p^{\beta}k'(1+p^{r-\eta+u}s)+ p^{r-\eta+u}s = p^{\beta}k'+p^{\beta+r-\eta+u}k's+ p^{r-\eta+u}s\\
\shortintertext{and}
& p^{q'}k''' = ba-a+1 = p^{\beta}k'(1+p^{r-\eta+u}s)- p^{r-\eta+u}s = p^{\beta}k'+p^{\beta+r-\eta+u}k's - p^{r-\eta+u}s,
\end{align*}
we obtain that
\begin{equation}\label{pepit4}
p^qk'' - p^{q'}k''' = 2 p^{r-\eta+u}s\quad\text{and}\quad p^qk'' + p^{q'}k''' = 2p^{\beta}k'+2p^{\beta+r-\eta+u}k's.
\end{equation}
From this it follows immediately that if $q\ne q'$, then $\beta = r-\eta+u = \min(q,q')$. So, we can assume that $q=q'$. Then the last equality in~\eqref{pepit4} implies that $q=q'\le \beta$, and hence $q=q'<\beta$ by assumption. Consequently, again by the last equality in~\eqref{pepit4}, we obtain that $p\mid k''+k'''$. Therefore $p\nmid k''-k'''$ (because otherwise $p\mid k''$), and hence, by the first equality in~\eqref{pepit4}, we have $q=q'=r-\eta+u$. This finishes the proof of the claim. Note moreover that, by~\eqref{pepit2}
\begin{equation}\label{pepit3}
p^{\eta-q+q'}\kappa k'''= c(ba-a+1) = c(ba+a-1) - 2c(a-1) = c(b+a-1) - 2c(a-1) = c(b-a+1),
\end{equation}
which implies that if $\eta-q+q'<r$, then $q'$ is the largest~inte\-ger lesser or equal than $r$, such that
\begin{equation*}\label{pepito4'''}
p^{q'}\mid b-a+1.
\end{equation*}
This also is true when $\eta-q+q'\ge r$. Indeed, in this case $q'>q$, and hence $\beta+r-\eta+u = 2\beta \ge r$. Consequently
$$
b-a+1 = p^{\beta}k' - p^{r-\eta + u}s = ba-a+1 = p^{q'}k''',
$$
as desired. Note finally that, since $\eta<r$, from~\eqref{tres casos} it follows that $u<\min(q,q',\beta)$.
\end{remark}

\begin{remark}\label{cond si q<eta} Since $r-\eta\le \beta$, we have
$$
p^{r-\eta}\mid p^{\beta}k'(1+p^{r-\eta} k)+p^{r-\eta} k = ba+a-1\quad\text{and}\quad p^{r-\eta}\mid p^{\beta}k'(1+p^{r-\eta} k)-p^{r-\eta} k = ba-a+1.
$$
Hence $r-\eta\le \min(q,q',\beta)$. Consequently, if $q<\eta$, then $r<2\eta$.
\end{remark}

\smallskip

In the following theorem we compute the extensions classes of $H$ by $I$.

\begin{theorem}\label{prop 4.10'} Under the hypothesis established at the beginning of this subsection three distinct cases arise:
\begin{description}[font=\normalfont\scshape, leftmargin=0cm]

\item[\underline{$p=2$}] Then $(b,a)\in\{(0,1),(0,3),(2,1),(2,3)\}$ and the following facts hold:

\smallskip

\begin{itemize}[topsep=0pt,itemsep=1.7pt]

\item[-] When $(b,a) = (0,1)$, a complete family $(\alpha_{\gamma},f_{\mathfrak{f}_0})$, of representatives of $2$-cocycles modulo coboundaries in~the~com\-plex $\wh{\mathcal{C}}_N^*(H,I)$, is parameterized by the pairs $(\gamma,\mathfrak{f}_0)\in\{(0,0),(0,2),(1,0),(1,2)\}$.

\item[-] When $(b,a) = (0,3)$, a complete family $(\alpha_{\gamma},f_{\mathfrak{f}_0})$, of representatives of $2$-cocycles modulo coboundaries in~the~com\-plex $\wh{\mathcal{C}}_N^*(H,I)$, is parameterized by the pairs $(\gamma,\mathfrak{f}_0)\in\{(0,0),(2,0),(1,1),(3,1)\}$.

\item[-] When $(b,a) = (2,1)$, a complete family $(\alpha_{\gamma},f_{\mathfrak{f}_0})$, of representatives of $2$-cocycles modulo coboundaries in~the~com\-plex $\wh{\mathcal{C}}_N^*(H,I)$, is parameterized by the pairs $(\gamma,\mathfrak{f}_0)\in\{(0,0),(2,0)\}$.

\item[-] When $(b,a) = (2,3)$, a complete family $(\alpha_{\gamma},f_{\mathfrak{f}_0})$, of representatives of $2$-cocycles modulo coboundaries in~the~com\-plex $\wh{\mathcal{C}}_N^*(H,I)$, is parameterized by the pairs $(\gamma,\mathfrak{f}_0)\in\{(0,0),(2,2),(1,1),(3,3)\}$.

\end{itemize}

\medskip

\item[\underline{$p$ odd and $\eta\le q$}] This case is divided in two subcases, namely $q'\le \beta$ and $q'>\beta$.

\smallskip

\begin{itemize}[topsep=0pt,itemsep=1.7pt]

\item[-] When $q'\le \beta$, a complete family $(\alpha_{\gamma},f_{\mathfrak{f}_0})$, of representatives of $2$-cocycles modulo coboundaries in $\wh{\mathcal{C}}_N^*(H,I)$, is~pa\-ram\-eterized by the pairs
$$
\qquad (\gamma,\mathfrak{f}_0)\in \{(p^{r-q'}z_1,p^{r-\eta}z_2)\in \mathbb{Z}_{p^r}\oplus \mathbb{Z}_{p^r}:0\le z_1<p^{\eta-r+q'}\text{ and } 0\le z_2<p^{\eta}\}.
$$

\item[-] When $q' > \beta$, we have $a\ne 1$, $q=\beta$ and a complete family $(\alpha_{\gamma},f_{\mathfrak{f}_0})$, of representatives of $2$-cocycles modulo~co\-boundaries in $\wh{\mathcal{C}}_N^*(H,I)$,~is~pa\-ram\-eterized by the pairs
$$
\qquad (\gamma,\mathfrak{f}_0)\in \{(p^{r-q'}z_1,p^{r-q'+\beta-\eta}(k''-k')z_1 + p^{r-\eta} z_2)\in \mathbb{Z}_{p^r}\oplus \mathbb{Z}_{p^r}: 0\le z_1<p^{\eta-r+q'} \text{ and } 0\le z_2 < p^{\eta}\}.
$$

\end{itemize}

\medskip

\item[\underline{$p$ odd and $q<\eta$}] This case is divided in five subcases: $q=q'=\beta$, $q=q'<\beta$, $q'=\beta<q$, $q=\beta<q'\le r-\eta +\beta$ and $q=\beta<q'>r-\eta +\beta$.

\smallskip

\begin{itemize}[topsep=0pt,itemsep=1.7pt]

\item[-] When $q=q'= \beta$, a complete family $(\alpha_{\gamma},f_{\mathfrak{f}_0})$, of representatives of $2$-cocycles modulo coboundaries in $\wh{\mathcal{C}}_N^*(H,I)$, is pa\-ram\-eterized by the pairs
$$
\qquad (\gamma,\mathfrak{f}_0)\in \{(p^{r-q'}z_1,p^{r-\eta}z_2)\in \mathbb{Z}_{p^r}\oplus \mathbb{Z}_{p^r}:0\le z_1<p^{q'}\text{ and } 0\le z_2<p^{\beta+\eta-r}\}.
$$

\item[-] When $q=q'<\beta$, we have $a\ne 1$, $r-q= \eta-u$ and a complete family $(\alpha_{\gamma},f_{\mathfrak{f}_0})$, of representatives of $2$-cocycles modulo coboundaries in $\wh{\mathcal{C}}_N^*(H,I)$, is pa\-ram\-eterized by the pairs
$$
\qquad (\gamma,\mathfrak{f}_0)\in \{(p^{r-q'}z_1,p^{r-\eta}z_2)\in \mathbb{Z}_{p^r}\oplus \mathbb{Z}_{p^r}:0\le z_1<p^{q'}\text{ and } 0\le z_2<p^u\}.
$$

\item[-] When $q'=\beta<q$, then $a\ne 1$, $r-q' = \eta-u$ and a complete family $(\alpha_{\gamma},f_{\mathfrak{f}_0})$, of representatives of $2$-cocycles modulo coboundaries in $\wh{\mathcal{C}}_N^*(H,I)$, is pa\-ram\-eterized by the pairs
$$
\qquad (\gamma,\mathfrak{f}_0)\in \{(p^{r-q'}z_1,p^{r-\eta}z_2)\in \mathbb{Z}_{p^r}\oplus \mathbb{Z}_{p^r}:0\le z_1<p^{q'}\text{ and } 0\le z_2<p^{q-r+\eta}\}.
$$

\item[-] When $q=\beta<q'$ and $q'\le r-\eta+\beta$, then $a\ne 1$, $r-q = \eta-u$ and a complete family $(\alpha_{\gamma},f_{\mathfrak{f}_0})$, of representatives of $2$-cocycles modulo coboundaries in $\wh{\mathcal{C}}_N^*(H,I)$, is pa\-ram\-eterized by the pairs
$$
\qquad (\gamma,\mathfrak{f}_0)\in \{(p^{r-q'}z_1,p^{r+\beta-q'-\eta}k'z_1 + p^{r-\eta}z_2)\in \mathbb{Z}_{p^r}\oplus \mathbb{Z}_{p^r}:\text{$0\le z_1<p^{\eta-r+q'}$ and $0\le z_2<p^q$}\}.
$$

\item[-] When $q=\beta<q'$ and $q' > r-\eta+\beta$, then $a\ne 1$, $r-q = \eta-u$ and a complete family $(\alpha_{\gamma},f_{\mathfrak{f}_0})$, of representatives of $2$-cocycles modulo coboundaries in $\wh{\mathcal{C}}_N^*(H,I)$, is pa\-ram\-eterized by the pairs
$$
\qquad (\gamma,\mathfrak{f}_0)\in \{(p^{\eta-\beta}z_1,k'z_1 + p^{r-\eta} z_2)\in \mathbb{Z}_{p^r}\oplus \mathbb{Z}_{p^r}:\text{$0\le z_1,z_2<p^{\beta}$}\}.
$$
\end{itemize}

\end{description}
\end{theorem}

\begin{proof} Let $a,b\in \mathbb{Z}_{p^r}$ satisfying~\eqref{pep2pr}. Our aim is to determine a complete set of representatives of $\ker T_1\cap \ker T_2$~mod\-ule~$\ima S$. For the sake of brevity, when defining the set of all $(x,y)\in \mathbb{Z}_{p^r}\oplus \mathbb{Z}_{p^r}$ that satisfy a property $P$, we will denote it by $\{(x,y):P\}$ rather than $\{(x,y)\in \mathbb{Z}_{p^r}\oplus \mathbb{Z}_{p^r}:P\}$. We divide the proof in three cases.

\smallskip

\noindent\underline{\textsc{First case:} $p=2$}\enspace An straightforward computation using~\eqref{T_1, T_2 y S} proves that
\begin{align*}
& T_1(\gamma,f_0) = \begin{cases} 2f_0 &\text{if $a=1$,}\\2f_0-2\gamma &\text{if $a=3$,} \end{cases}\\
& T_2(\gamma,f_0) = \begin{cases} (b+2)f_0-b\gamma &\text{if $a=1$,}\\bf_0-b\gamma &\text{if $a=3$,} \end{cases}
\shortintertext{and}
& S(t) = \begin{cases} (2t,bt) &\text{if $a=1$,}\\ (2t,bt+2t) &\text{if $a=3$.}\end{cases}
\end{align*}
Therefore,
$$
\ker(T_1)\cap \ker(T_2) = \begin{cases} \mathbb{Z}_4\oplus 2\mathbb{Z}_4 &\text{if $a=1$ and $b=0$,}\\ 2\mathbb{Z}_4\oplus 2\mathbb{Z}_4 = \langle (2,2) \rangle\oplus \langle (2,0) \rangle &\text{if $a=1$ and $b=2$,}\\ \{(\gamma,f_0):2\mid f_0-\gamma\} = \langle (1,1) \rangle\oplus \langle (2,0) \rangle  &\text{if $a=3$}\end{cases}
$$
and
$$
\ima(S) = \begin{cases} 2\mathbb{Z}_4\oplus 0 &\text{if $a=1$ and $b=0$,}\\ \langle (2,2) \rangle  &\text{if $a=1$ and $b=2$,}\\ \langle (2,2) \rangle &\text{if $a=3$ and $b=0$,}\\ \langle (2,0) \rangle  &\text{if $a=3$ and $b=2$.}\end{cases}
$$
Hence,
\begin{equation}\label{como p=2 eta=1 r=2}
\Ho^2_{\blackdiamond,\Yleft}(H,I)\simeq \begin{cases} \frac{\mathbb{Z}_4}{2 \mathbb{Z}_4}\oplus 2\mathbb{Z}_4 = \frac{\langle (1,0)\rangle}{\langle (2,0)\rangle} \oplus \langle (0,2)\rangle &\text{if $a=1$ and $b=0$,}\\ \langle (2,0)\rangle &\text{if $a=1$ and $b=2$,}\\ \frac{\langle (1,1)\rangle}{\langle (2,2)\rangle} \oplus \langle (2,0)\rangle &\text{if $a=3$ and $b=0$,} \\ \langle (1,1)\rangle  &\text{if $a=3$ and $b=2$.}\end{cases}
\end{equation}
The assertions for $p=2$, follow immediately from this fact.

\smallskip

\noindent\underline{\textsc{Second case:} $p$ odd and $\eta\le q$}\enspace By Remark~\ref{caso eta le q}, we are in case~b) of Subsubsection~\ref{two particular cases} with
\[
d \coloneqq p^{\eta} \qquad\text{and}\qquad c \coloneqq p^{q-\eta}k''.
\]
Let~$\wt{T}_1$, $\wt{T}_2$ and~$\wt{S}$ be the maps define by
$$
\wt{T}_1\coloneqq T_1\xcirc \Phi,\quad \wt{T}_2\coloneqq T_2\xcirc \Phi \quad\text{and}\quad  \wt{S}\coloneqq \Phi^{-1}\xcirc S,
$$
where $\Phi\colon I^2\to I^2$ is the map given by $\Phi(x,y)\coloneqq (x,cx+y)$. As we saw in~\eqref{Phi barra'}, the map $\Phi$ induces an iso\-morphism
\begin{equation}\label{Phi barra}
\ov{\Phi}\colon \frac{\ker \wt{T}_1\cap \ker \wt{T}_2}{\ima \wt{S}}\longrightarrow \frac{\ker T_1\cap \ker T_2}{\ima S}\simeq \Ho^2_{\blackdiamond,\Yleft}(H,I).
\end{equation}
Moreover, by~\eqref{pepito5}, we have
\begin{equation}\label{wt(S), wt{T}_1 y wt{T}_2}
\wt{T}_1(x,y)= bx+p^{\eta}y,\quad\wt{T}_2(x,y)= \sum_{\ell=0}^{p^{\eta}-1} \binom{p^{\eta}}{\ell+1} c^{\ell} p^{\eta \ell}(y+cx)- bx \quad \text{and}\quad \wt{S}(t) = (p^{\eta}t,0)
\end{equation}
We claim that
\begin{equation}\label{pepito1'}
\wt{T}_2(x,y) = - bx + p^{\eta}cx + p^{\eta}y.
\end{equation}
Since $c=p^{q-\eta}k''$, in order to verify the claim in this case, it suffices to show that
$$
p^r\mid \binom{p^{\eta}}{\ell+1} p^{q\ell}\qquad\text{for $\ell\ge 1$.}
$$
For $\ell = 1$ this holds since $p^{\eta}\mid \binom{p^{\eta}}{2}$ and $q\ge r-\eta$ (see~\eqref{un caso'} and~\eqref{tres casos'}); while, for $\ell \ge 2$, this follows from the fact that $q\ge r-\eta$ and $q\ge \eta $, which implies that $q\ell \ge r-\eta +\eta \ge r$. We next compute
$$
K\coloneqq \ker\wt{T}_1\cap \ker\wt{T}_2.
$$
By~\eqref{wt(S), wt{T}_1 y wt{T}_2} and~\eqref{pepito1'},
$$
(x,y)\in K \Longleftrightarrow bx = -p^{\eta}y \text{ and } bx = p^{\eta}y + p^{\eta} cx \Longleftrightarrow p^{\eta}y = -p^{\beta}k'x \text{ and } p^qk'' x - 2p^{\beta}k'x = 0.
$$
Hence, by~\eqref{pepito3'N},
\begin{equation}\label{pepito6'}
(x,y)\in K \Longleftrightarrow p^{\eta}y=-p^{\beta}k'x \text{ and } p^{q'}x = 0 \Longleftrightarrow x\in p^{r-q'}\mathbb{Z}_{p^r}\text{ and } p^{\eta}y = -p^{\beta}k'x.
\end{equation}
Consequently, if $q'\le \beta$, then
\begin{equation}\label{pepito4'}
(x,y)\in K \Longleftrightarrow x\in p^{r-q'} \mathbb{Z}_{p^r} \text{ and } p^{\eta} y = 0 \Longleftrightarrow (x,y)\in p^{r-q'} \mathbb{Z}_{p^r} \oplus p^{r-\eta} \mathbb{Z}_{p^r}.
\end{equation}
Since $\ima \wt{S}\subseteq K$, this implies that $r-q'\le\eta$ and
\begin{equation}\label{para homologia 1}
\frac{K}{\ima \wt{S}} = \frac{p^{r-q'} \mathbb{Z}_{p^r}}{p^{\eta}\mathbb{Z}_{p^r}}  \oplus p^{r-\eta} \mathbb{Z}_{p^r}.
\end{equation}
Therefore,
$$
K'\coloneqq \{(p^{r-q'}z_1,p^{r-\eta}z_2):0\le z_1<p^{\eta-r+q'}\text{ and } 0\le z_2<p^{\eta}\}
$$
is a complete family of rep\-resentatives of $K$ module $\ima \wt{S}$. Hence
$$
\Phi(K') = \{(p^{r-q'}z_1,p^{r-q'+q-\eta}k'' z_1 + p^{r-\eta}z_2):0\le z_1<p^{\eta-r+q'}\text{ and } 0\le z_2<p^{\eta}\}
$$
is a complete family of representatives of $\ker T_1\cap \ker T_2$ module $\ima S$. But, by conditions~\eqref{un caso'} and~\eqref{tres casos'}, the~in\-equality $q' \le \beta$ immediately yields $q' \le q$, which implies $r-q'+q-\eta \ge r-\eta$. Consequently, $\Phi(K') = K'$. The assertion in the case $p$ odd, $\eta\le q$ and $q'\le \beta$, follows directly from this fact. It remains to consider the case $\beta < q'$, which by Remark~\ref{caso eta le q} implies that $a\ne 1$ and $r-\eta+u=\beta=q$. Consequently,
\begin{equation}\label{pep3}
0\le r-q' < r-\beta \le r-\beta+u = \eta.
\end{equation}
Take $(x,y)\in K$. By condition~\eqref{pepito6'} there exists $x'\in \mathbb{Z}_{p^r}$ such that $x = p^{r-q'}x'$, and we have
\begin{equation}\label{pepito2'}
p^{\eta} y = -p^{\beta}k'x = - p^{r-q'+\beta} k'x'.
\end{equation}
Since $r-q'+\beta\ge\beta \ge \eta$, this happens if and only if $y\in - p^{r-q'+\beta-\eta} k'x'+ p^{r-\eta} \mathbb{Z}_{p^r}$. Therefore, again by~\eqref{pepito6'},
\begin{equation}\label{pepito5'}
(x,y)\in K \Longleftrightarrow x = p^{r-q'} x'\text{ and } y = - p^{r-q'+\beta-\eta} k'x' + p^{r-\eta} y' \quad \text{for some $0\le x'<p^{q'}$ and $0\le y'<p^{\eta}$}.
\end{equation}
Note that $|K| = p^{q'+\eta}$. Take $(x_1,y_1),(x_2,y_2)\in K$ and write
$$
x_1 = p^{r-q'} x'_1,\quad x_2 = p^{r-q'} x'_2,\quad y_1 = -p^{r-q'+\beta-\eta}k' x'_1 + p^{r-\eta}y'_1\quad\text{and}\quad y_2 = -p^{r-q'+\beta-\eta}k' x'_2 + p^{r-\eta}y'_2,
$$
with $0\le x_1',x_2'<p^{q'}$ and $0\le y'_1,y'_2<p^{\eta}$. We assert that the following facts are equivalent:

\begin{enumerate}

\item $(x_1,y_1)\equiv (x_2,y_2)\mod{\ima\wt{S}}$.

\item There exists $z\in \mathbb{Z}_{p^r}$ such that $x'_2-x'_1 = p^{\eta-r+q'} z$ and $y'_2-y'_1 - p^u k' z \in  p^{\eta} \mathbb{Z}_{p^r}$.

\item $x'_2-x'_1 \in p^{\eta-r+q'} \mathbb{Z}_{p^r}$, and for all $z$ such that $x'_2-x'_1 = p^{\eta-r+q'} z$, we have $y'_2-y'_1 - p^u k' z \in  p^{\eta} \mathbb{Z}_{p^r}$.

\end{enumerate}
In fact, $(x_1,y_1)\equiv (x_2,y_2)\mod{\ima\wt{S}}$ if and only if $x_2 - x_1 \in  p^{\eta}\mathbb{Z}_{p^r}$ and $y_2=y_1$. Note also that, by~\eqref{pep3},
$$
x_2 - x_1 \in  p^{\eta}\mathbb{Z}_{p^r} \Longleftrightarrow x'_2-x'_1 \in p^{\eta-r+q'} \mathbb{Z}_{p^r} + p^{q'}\mathbb{Z}_{p^r} = p^{\eta-r+q'} \mathbb{Z}_{p^r},
$$
and write $x'_2-x'_1 = p^{\eta-r+q'} z$ with $z\in \mathbb{Z}_{p^r}$. Then
$$
y_2 = -p^{r-q'+\beta-\eta}k' x'_2 + p^{r-\eta}y'_2 = -p^{r-q'+\beta-\eta}k' x'_1 -p^{\beta}k'z + p^{r-\eta}y'_2,
$$
and so
$$
0 = y_1-y_2 = p^{\beta} k' z + p^{r-\eta} (y'_1-y'_2) = p^{r-\eta} (p^u k' z + y'_1-y'_2),
$$
where the last equality holds, since $\beta = r-\eta+u$. But this happens if and only if $y'_2-y'_1 - p^u k' z \in  p^{\eta} \mathbb{Z}_{p^r}$. Thus 1) implies 3). It is clear that~3) implies~2). Assume now that~2) is true and write $y'_2 = y'_1 + p^uk'z+ p^{\eta}z'$ with $z'\in \mathbb{Z}_{p^r}$. Then
\begin{align*}
& x_2-x_1 = p^{r-q'}p^{\eta-r+q'} z = p^{\eta}z\\
\shortintertext{and}
& y_2 = -p^{r-q'+\beta-\eta}k' x'_2 + p^{r-\eta}y'_2\\
&\phantom{y_2} = -p^{r-q'+\beta-\eta}k' x'_1 - p^{\beta}k' z + p^{r-\eta}y'_1  + p^{r-\eta+u} k'z\\
&\phantom{y_2} = -p^{r-q'+\beta-\eta}k' x'_1 + p^{r-\eta}y'_1 &&\text{since $\beta = r-\eta+u$}\\
&\phantom{y_2} = y_1,
\end{align*}
which proves that~2) implies~1). We claim that
\begin{equation}\label{para despues0}
K'\coloneqq \{(x,y): x = p^{r-q'}x'\text{ and } y = -p^{r-q'+\beta-\eta}k'x' + p^{r-\eta} y'\quad\text{with $0\le x'<p^{\eta-r+q'}$ and $0\le y'<p^{\eta}$}\},
\end{equation}
is a complete family of represents of $K$, module $\ima\wt{S}$. Since $|K| = p^{q'+\eta}$ and $|\ima \wt{S}| = p^{r-\eta}$, we know that
$$
\left|\frac{K}{\ima\wt{S}}\right| = p^{q'+2\eta-r} = |K'|.
$$
Hence, the claim follows, because, by item~3), if \hbox{$(x_1,y_1),(x_2,y_2)\in K'$} are equivalent module $\ima \wt{S}$, then they coincide. Consequently, since $c = p^{q-\eta}k'' = p^{\beta-\eta}k''$, we conclude that
$$
\Phi(K') = \{(p^{r-q'}z_1,p^{r-q'+\beta-\eta}(k''-k') z_1 + p^{r-\eta} z_2): 0\le z_1<p^{\eta-r+q'} \text{ and } 0\le z_2 < p^{\eta}\}
$$
is a complete family of representatives of $\ker T_1\cap \ker T_2$ module $\ima S$. Hence, the assertion in the case $p$ odd, $\eta\le q$ and $\beta<q'$, is true.

\smallskip

\noindent\underline{\textsc{Third case:} $p$ odd and $q<\eta$}\enspace By Remark~\ref{caso q < eta}, we are in the case~a)~of~Sub\-sub\-section~\ref{two particular cases}, with
\begin{equation}\label{para despues3}
d\coloneqq p^{\eta}\quad\text{and}\quad c\coloneqq p^{\eta-q}\kappa,
\end{equation}
where $\kappa\in \mathbb{Z}_{p^r}$ satisfies $p^{\eta}\kappa k''=p^{\eta}$. Let $\wt{T}_1$, $\wt{T}_2$ and~$\wt{S}$ be the maps define by
$$
\wt{T}_1\coloneqq T_1\xcirc \Phi,\quad \wt{T}_2\coloneqq T_2\xcirc \Phi \quad\text{and}\quad  \wt{S}\coloneqq \Phi^{-1}\xcirc S,
$$
where $\Phi\colon I^2\to I^2$ is the map given by $\Phi(x,y)\coloneqq (x+cy,y)$. As we saw in~\eqref{Phi barra1'}, the map $\Phi$ induces an iso\-morphism
\begin{equation}\label{Phi barra1''}
\ov{\Phi}\colon \frac{\ker \wt{T}_1\cap \ker \wt{T}_2}{\ima \wt{S}}\longrightarrow \frac{\ker T_1\cap \ker T_2}{\ima S}\simeq \Ho^2_{\blackdiamond,\Yleft}(H,I).
\end{equation}
Furthermore, by~\eqref{eq2}--\eqref{eq4}, \eqref{eq5} and Remark~\ref{remark 1.11'}, we have
\begin{equation}\label{eq2'}
\wt{T}_1(x,y)=(1-a)x+cby,\quad \wt{T}_2(x,y)= b\sum_{j=1}^{p^{\eta}-1} j a^jy+N(a)y-cby-bx\quad\text{and}\quad \wt{S}(t)=(0,p^qk''t).
\end{equation}
Arguing as in the proof of~\eqref{cociclo1 ej 4'} and using that $\beta+\eta\ge r$, we obtain
$$
N(a)+b\sum_{j=1}^{p^{\eta}-1} j a^j = \sum_{\ell=0}^{p^{\eta}-1} \binom{p^{\eta}}{\ell+1} p^{(r-\eta)\ell}k^{\ell} + p^{\beta}k'\sum_{\ell=1}^{p^{\eta}-1}\left((p^{\eta}-1) \binom{p^{\eta}}{\ell+1} - \binom{p^{\eta}}{\ell+2}\right) p^{(r-\eta)\ell}k^{\ell}.
$$
Write $\ell+1 = p^tv$ with $p\nmid v$ and $t\ge 0$. By Remark~\ref{valuacion de p en combinatorio}, we known that $p^{\eta-t}\mid \binom{p^{\eta}}{\ell+1}$. Therefore, the exponent $w$, of~$p$ in $\binom{p^{\eta}}{\ell+1} p^{(r-\eta)\ell}$, is at least $\eta-t+(r-\eta)\ell$. Thus, if $\ell\ge 1$, then
\begin{equation*}
w\ge \eta-t+(r-\eta)\ell = r-t+(r-\eta)(\ell-1)\ge r-t + \ell-1\ge r,
\end{equation*}
where we have used that $r-\eta>0$ and that $\ell-1 = p^tv-2\ge t$, for $\ell\ge 1$. Hence,
$$
N(a)+b\sum_{j=1}^{p^{\eta}-1} j a^j = p^{\eta} - p^{\beta}k'\sum_{\ell=1}^{p^{\eta}-1} \binom{p^{\eta}}{\ell+2} p^{(r-\eta)\ell}k^{\ell}.
$$
Writing now $\ell+2 = p^tv$ with $p\nmid v$ and $t\ge 0$, we obtain that the exponent $w$, of $p$ in $\binom{p^{\eta}}{\ell+2} p^{\beta+(r-\eta)\ell}$, is greater than or equal to $\eta-t+\beta+(r-\eta)\ell$. Since $\beta\ge r-\eta\ge 1$ and $\ell\ge 1$, we obtain
\begin{equation*}
w\ge \eta-t+\beta+(r-\eta)\ell = r-t+\beta+(r-\eta)(\ell-1)\ge r-t+\beta + \ell-1\ge r-t + \ell\ge r,
\end{equation*}
where we have used that $r-\eta>0$ and that $\ell = p^tv-2 \ge t$. Consequently,
$$
N(a)+b\sum_{j=1}^{p^{\eta}-1} j a^j = p^{\eta}.
$$
Combining this with~\eqref{eq2'}, we conclude that
\begin{equation}\label{eq3'}
\wt{T}_1(x,y)= cby-(a-1)x,\quad \wt{T}_2(x,y)=p^{\eta}y-cby-bx\quad\text{and}\quad \wt{S}(t)=(0,p^qk''t).
\end{equation}
Since $cb = p^{\eta}-c(a-1)$, we have
\begin{equation}\label{eq4'}
\wt{T}_1(x,y)= p^{\eta}y-(a-1)(cy+x)\qquad\text{and}\qquad \wt{T}_2(x,y)= p^{\eta}y-b(cy+x),
\end{equation}
which implies
\begin{equation}\label{eq5'}
\wt{T}_1(x,y)- \wt{T}_2(x,y)= (b-a+1)(cy+x).
\end{equation}
We next compute
$$
K\coloneqq \ker\wt{T}_1\cap \ker\wt{T}_2.
$$
Recall, from Remark~\ref{caso q < eta}, that $q$ and $q'$ are the largest~inte\-gers lesser or equal than $r$, such that
$$
p^q\mid b+a-1\qquad\text{and}\qquad p^{q'}\mid b-a+1.
$$
Assume first that
$$
q=q'=\beta
$$
and write
$$
b-a+1 = b k^{\mathrm{iv}} \quad\text{and}\quad b+a-1 = bk^{\mathrm{v}}\qquad\text{with $k^{\mathrm{iv}}$ and $k^{\mathrm{v}}$ units of $\mathbb{Z}_{p^r}$.}
$$
Hence, by equality~\eqref{eq5'}, we deduced that $(x,y)\in \ker(\wt{T}_1-\wt{T}_2)$ if and only if $b(cy+x) = 0$. Combining this with the second equality in~\eqref{eq4'}, we obtain that $(x,y)\in K$ if and only if
$$
b(cy+x) = 0 \qquad\text{and}\qquad p^{\eta}y = 0.
$$
Since $p^{\eta} = c(ba+a-1) = c(b+a-1) = bck^{\mathrm{v}}$, this happens if and only if $p^{\eta}y = 0$ and $bx = 0$. Thus,
$$
\ker\wt{T}_1\cap \ker\wt{T}_2 = p^{r-\beta}\mathbb{Z}_{p^r}\oplus p^{r-\eta}\mathbb{Z}_{p^r}.
$$
Hence, by the last equality in~\eqref{eq3'}, we have
\begin{equation}\label{para homologia 2}
\frac{\ker\wt{T}_1\cap \ker\wt{T}_2}{\ima{\wt{S}}} = p^{r-\beta}\mathbb{Z}_{p^r}\oplus \frac{p^{r-\eta}\mathbb{Z}_{p^r}}{p^{\beta}\mathbb{Z}_{p^r}}.
\end{equation}
Thus,
$$
K'\coloneqq \{(p^{r-\beta}z_1,p^{r-\eta}z_2):0\le z_1<p^{\beta}\text{ and } 0\le z_2<p^{\beta+\eta-r}\}
$$
is a complete family of rep\-resentatives of $K$ module $\ima \wt{S}$. Consequently
$$
\Phi(K') = \{(p^{r-\beta}z_1+p^{r-q}\kappa z_2,p^{r-\eta}z_2):0\le z_1<p^{\beta}\text{ and } 0\le z_2<p^{\beta+\eta-r}\}
$$
is a complete family of representatives of $\ker T_1\cap \ker T_2$ module $\ima S$. Since $q = \beta$, we have $\Phi(K') = K'$. The~asser\-tion in the case $p$ odd and $q=q'=\beta$, follows from this fact. Suppose now that
$$
q\ne q',\qquad q\ne \beta\qquad\text{or}\qquad q'\ne \beta.
$$
By Remark~\ref{caso q < eta}, this implies that $a\ne 1$ and one of the conditions in~\eqref{tres casos} is satisfied. Assume first that
$$
q = q' = r-\eta +u<\beta
$$
and write
$$
b-a+1 = (a-1) k^{\mathrm{iv}}\quad\text{and}\quad b+a-1 = (a-1)k^{\mathrm{v}}\qquad\text{with $k^{\mathrm{iv}}$ and $k^{\mathrm{v}}$ units of $\mathbb{Z}_{p^r}$.}
$$
By equality~\eqref{eq5'}, we deduced that $(x,y)\in \ker(\wt{T}_1-\wt{T}_2)$ if and only if $(a-1)(cy+x) = 0$. Combining this with the first equality in~\eqref{eq4'}, it follows that $(x,y)\in K$ if and only if
$$
(a-1)(cy+x) = 0 \qquad\text{and}\qquad p^{\eta}y = 0.
$$
Since $p^{\eta} = c(b+a-1) = (a-1)ck^{\mathrm{v}}$, this happens if and only if $p^{\eta}y = 0$ and $(a-1)x = 0$. Thus,
$$
\ker\wt{T}_1\cap \ker\wt{T}_2 = p^{\eta-u}\mathbb{Z}_{p^r}\oplus p^{r-\eta}\mathbb{Z}_{p^r},
$$
and so, by the last equality in~\eqref{eq3'}, we have
\begin{equation}\label{para homologia 3}
\frac{\ker\wt{T}_1\cap \ker\wt{T}_2}{\ima{\wt{S}}} = p^{\eta-u}\mathbb{Z}_{p^r}\oplus \frac{p^{r-\eta}\mathbb{Z}_{p^r}}{p^{r-\eta +u}\mathbb{Z}_{p^r}}.
\end{equation}
Therefore,
$$
K'\coloneqq \{(p^{\eta-u}z_1,p^{r-\eta}z_2):0\le z_1<p^{r-\eta+u}\text{ and } 0\le z_2<p^u\}
$$
is a complete family of rep\-resentatives of $K$ module $\ima \wt{S}$. Consequently
$$
\Phi(K') = \{(p^{\eta-u}z_1+p^{r-q}\kappa z_2,p^{r-\eta}z_2):0\le z_1<p^{r-\eta+u}\text{ and } 0\le z_2<p^u\}
$$
is a complete family of representatives of $\ker T_1\cap \ker T_2$ module $\ima S$. Since $r-q = \eta-u$, we obtain that $\Phi(K') = K'$. The assertion in the case $p$ odd, $q<\eta$ and $q = q' <\beta$, follows from this fact. Next assume that
$$
q' =\beta = r-\eta+u < q,
$$
and write
$$
b-a+1 = bk^{\mathrm{iv}}\qquad\text{with $k^{\mathrm{iv}}$ a unit of $\mathbb{Z}_{p^r}$.}
$$
By equality~\eqref{eq5'} and the second equality in~\eqref{eq4'},
$$
(x,y)\in K \Longleftrightarrow b(cy+x) = 0\text{ and } p^{\eta} y = 0 \Longleftrightarrow y\in p^{r-\eta} \mathbb{Z}_{p^r}\text{ and } cy+x \in p^{r-\beta} \mathbb{Z}_{p^r}
$$
Consequently, the map
$$
\Psi\colon \mathbb{Z}_{p^{\beta}}\oplus \mathbb{Z}_{p^{\eta}}\to K,\quad\text{given by}\quad \Psi(z_1,z_2)\coloneqq  (p^{r-\beta}z_1-cp^{r-\eta}z_2,p^{r-\eta}z_2),
$$
is an isomorphism. By the last equality in~\eqref{eq3'}, we know that $\ima \wt{S} = 0\oplus p^q \mathbb{Z}_{p^r}$. Hence,
\begin{equation}\label{para despues1}
K'\coloneqq \{(x,y): x = p^{r-\beta}z_1-cp^{r-\eta}z_2\text{ and } y = p^{r-\eta}z_2\quad\text{with $0\le z_1<p^{\beta}$ and $0\le z_2<p^{q-r+\eta}$}\}.
\end{equation}
is a complete family~of~rep\-resentatives of $K$ module $\ima \wt{S}$. Consequently
$$
\Phi(K') = \{(x,y): x = p^{r-\beta}z_1 \text{ and } y = p^{r-\eta}z_2\quad\text{with $0\le z_1<p^{\beta}$ and $0\le z_2<p^{q-r+\eta}$}\}
$$
is a complete family of representatives of $\ker T_1\cap \ker T_2$ module $\ima S$. Hence, the assertion in the case $p$ odd, $q<\eta$ and $q' =\beta < q$, is true. Assume finally that
$$
q =\beta = r-\eta+u < q'.
$$
Again by equality~\eqref{eq5'} and the second equality in~\eqref{eq4'},
\begin{equation}\label{ppal0}
(x,y)\in K \Longleftrightarrow p^{q'}(cy+x) = 0\text{ and } p^{\eta}y-p^{\beta}k'(cy+x) = 0.
\end{equation}
Note that this happens if and only if
\begin{equation}\label{ppal1}
cy+x = p^{r-q'}x'\quad\text{and}\quad p^{\eta}y-p^{r-q'+\beta}k'x' = 0, \qquad\text{for some $0\le x'<p^{q'}$}.
\end{equation}
Suppose $\eta\le r-q'+\beta$. Then
$$
p^{\eta}y-p^{r-q'+\beta}k'x' = 0 \Longleftrightarrow y-p^{r+\beta-q'-\eta}k'x'\in p^{r-\eta}\mathbb{Z}_{p^r}.
$$
Hence, in this case
\begin{equation}\label{pepe1}
K = \{(x,y): y = p^{r+\beta-q'-\eta}k'x' + p^{r-\eta}y' \text{ and } x = p^{r-q'}x'-cy\quad\text{with $0\le x'<p^{q'}$ and $0\le y'<p^{\eta}$}\}.
\end{equation}
Note that $|K|= p^{q'+\eta}$. We claim that
\begin{equation}\label{para despues2}
K'\coloneqq \{(x,y): y = p^{r+\beta-q'-\eta}k'x' + p^{r-\eta}y' \text{ and } x = p^{r-q'}x'-cy\quad\text{with $0\le x'<p^{\eta-r+q'}$ and $0\le y'<p^q$}\}
\end{equation}
is a complete family of represents of $K$, module $\ima\wt{S}$. Since $|K| = p^{q'+\eta}$ and $|\ima \wt{S}| = p^{r-q}$, we know that
$$
\left|\frac{K}{\ima\wt{S}}\right| = p^{q'+\eta+q-r} = |K'|.
$$
Therefore, to prove the claim it suffices to verify that if
$$
(x_1,y_1), (x_2,y_2)\in K'\qquad\text{and}\qquad (x_1,y_1) - (x_2,y_2) = (0,y)\quad\text{with $(0,y)\in\ima \wt{S}$,}
$$
then they coincide. Write
$$
x_1 = p^{r-q'} x'_1-cy_1,\quad x_2 = p^{r-q'} x'_2-cy_2,\quad y_1 = p^{r+\beta-q'-\eta}k' x'_1 + p^{r-\eta}y'_1\quad\text{and}\quad y_2 = p^{r+\beta-q'-\eta}k' x'_2 + p^{r-\eta}y'_2,
$$
with $0\le x_1',x_2'<p^{\eta-r+q'}$ and $0\le y'_1,y'_2<p^q$. By the last equality in~\eqref{eq3'}, we know that there exist $0\le t<p^r$ such that
\begin{align*}
& y = p^qk''t =  bat+at-t = p^{r+\beta-q'-\eta}k'(x'_1-x'_2) + p^{r-\eta}(y'_1-y'_2)
\shortintertext{and}
& 0 = p^{r-q'}(x'_1-x'_2)-c(y_1-y_2) = p^{r-q'}(x'_1-x'_2)-cy = p^{r-q'}(x'_1-x'_2)-p^{\eta}t,
\end{align*}
where we have used equality~\eqref{pepit2}. Since $\eta-r+q'>u\ge 0$ and $\eta\le r$, from the equality $p^{r-q'}(x'_1-x'_2)=p^{\eta}t$, we obtain that $p^{\eta-r+q'}\mid x'_1-x'_2$. Thus $x'_1 = x'_2$, because $0\le |x'_1-x'_2| < p^{\eta-r+q'}$. Consequently $p^{\eta}t = 0$ in $\mathbb{Z}_{p^r}$, and so there exists $0\le t'<p^{\eta}$ such that $t = p^{r-\eta}t'$. Therefore
$$
p^{r-\eta}(y'_1-y'_2) = y = p^qk''p^{r-\eta}t' = p^{r-\eta+q}k''t',
$$
and hence $p^q\mid y'_1-y'_2$ because $q\le \eta$. Since $0\le |y'_1-y'_2|<p^q$, this implies that $y=0$, which finishes the proof of the claim. So,
$$
\Phi(K') = \{(x,y): x = p^{r-q'}x' \text{ and } y = p^{r+\beta-q'-\eta}k'x' + p^{r-\eta}y'\quad\text{with $0\le x'<p^{\eta-r+q'}$ and $0\le y'<p^q$}\}
$$
is a complete family of representatives of $\ker T_1\cap \ker T_2$ module $\ima S$. This concludes the proof that the assertion in the case $p$ odd, $q<\eta$, $q =\beta < q'$ and $q'\le r-\eta+\beta$, is true. Suppose now that $r-q'+\beta<\eta$ and let $x$, $y$ and $x'$ be as in~\eqref{ppal0} and~\eqref{ppal1}. Note that,
$$
p^{\eta}y-p^{r-q'+\beta}k'x' = 0 \Longleftrightarrow p^{\eta-r-\beta+q'}y-k'x'\in p^{q'-\beta}\mathbb{Z}_{p^r}.
$$
Clearly the condition on the right side of this equivalence is satisfied if and only
\begin{equation}\label{pepe0}
k'x' = p^{\eta-r-\beta+q'}y- p^{q'-\beta}y',\qquad\text{for some $0\le y'<p^{r+\beta-q'}$},
\end{equation}
which implies that $x' \in  p^{\eta-r-\beta+q'} \mathbb{Z}_{p^r}$ because $k'\in \mathbb{Z}_{p^r}$ is invertible. Since, moreover $0\le x'<p^{q'}$, we can write
$$
x' = p^{\eta-r-\beta+q'}t\quad\text{with $0\le t < p^{r+\beta-\eta}$.}
$$
Since, by~\eqref{ppal1}
$$
cy+x = p^{\eta-\beta}t\qquad\text{and}\qquad p^{\eta}(y-k't) = p^{\eta}y-p^{r-q'+\beta}k'x' = 0,
$$
we have
\begin{equation}\label{pepe2}
K = \{(x,y): y = k't + p^{r-\eta} t' \text{ and } x = p^{\eta-\beta}t-cy\quad\text{with $0\le t<p^{r+\beta-\eta}$ and $0\le t'<p^{\eta}$}\},
\end{equation}
Note that $|K|= p^{r+\beta}$. We claim that
\begin{equation}\label{para despues4}
K'\coloneqq \{(x,y): y = k't + p^{r-\eta} t' \text{ and } x = p^{\eta-\beta}t-cy\quad\text{with $0\le t,t'<p^{\beta}$}\},
\end{equation}
is a complete family of represents of $K$, module $\ima\wt{S}$. Since $|K| = p^{r+\beta}$ and $|\ima \wt{S}| = p^{r-q}$, we know that
$$
\left|\frac{K}{\ima\wt{S}}\right| = p^{2\beta} = |K'|.
$$
Therefore, to prove the claim it suffices to verify that if
$$
(x_1,y_1), (x_2,y_2)\in K'\qquad\text{and}\qquad (x_1,y_1) - (x_2,y_2) = (0,y)\quad\text{with $(0,y)\in\ima \wt{S}$,}
$$
then they coincide. Write
$$
x_1 = p^{\eta-\beta} t_1-cy_1,\quad x_2 = p^{\eta-\beta} t_2-cy_2,\quad y_1 = k't_1 + p^{r-\eta}t'_1\quad\text{and}\quad y_2 = k't_2 + p^{r-\eta}t'_2,
$$
with $0\le t_1,t_2<p^{\beta}$ and $0\le t'_1,t'_2<p^q$. By the last equality in~\eqref{eq3'}, we know that there exist $t\in \mathbb{Z}_{p^r}$ such that
\begin{align*}
& y = p^qk''t = bat+at-t = k'(t_1-t_2) + p^{r-\eta}(t'_1-t'_2)
\shortintertext{and}
& 0 = p^{\eta-\beta}(t_1-t_2)-c(y_1-y_2) = p^{\eta-\beta}(t_1-t_2)-cy = p^{\eta-\beta}(t_1-t_2)-p^{\eta}t,
\end{align*}
where we have used~\eqref{pepit2}. Since $\eta\le r$, from the equality $p^{\eta-\beta}(t_1-t_2)=p^{\eta}t$, we obtain that $p^{\beta}\mid t_1-t_2$.~Consequent\-ly $t_1 = t_2$, because $0\le |t_1-t_2| < p^{\beta}$. Thus $p^{\eta}t = 0$ in $\mathbb{Z}_{p^r}$, and so there exists $0\le t'<p^{\eta}$ such that $t = p^{r-\eta}t'$. Therefore
$$
p^{r-\eta} (t'_1-t'_2) = y = bat+at-t = p^qk''p^{r-\eta}t' = p^{r-\eta+q}k''t'.
$$
Since $0\le |t'_1-t'_2|<p^q$, this implies that $y=0$, which finishes the proof of the claim. Hence
$$
\Phi(K') = \{(x,y): x = p^{\eta-\beta}t \text{ and } y = k't + p^{r-\eta}t'\quad\text{with $0\le t,t'<p^{\beta}$}\}
$$
is a complete family of representatives of $\ker T_1\cap \ker T_2$ module $\ima S$. This concludes the proof that the assertion in the case $p$ odd, $q<\eta$, $q =\beta < q'$ and $r-\eta+\beta<q'$, is true.
\end{proof}

\begin{remark} By Remark~\ref{cond si q<eta}, in the case $p$ odd and $q<\eta$, we have $r<2\eta$.
\end{remark}

\begin{remark} The formulas for $\blackdiamond$, $\Yleft$ and $f_{\mathfrak{f}_0}$ were presented in~\eqref{pep2pr}, \eqref{accion1 ej 4'} and~\eqref{cociclo1 ej 4'}.
Under appropriate assumptions, these expressions simplify substantially. For example, when $p=2$
$$
h\blackdiamond y = \begin{cases} y & \text{if $a=1$,}\\ 3^h y &\text{if $a=3$,}\end{cases}\qquad y\Yleft h = \begin{cases} 0 & \text{if $b=0$,}\\ 2hy &\text{if $b=2$}\end{cases} \qquad\text{and}\qquad f_{\mathfrak{f}_0}(h,h') = \mathfrak{f}_0hh',
$$
while, when $p$ is odd and $r\ge 2\eta$
$$
h\blackdiamond y = y + h p^{r-\eta} k y\qquad\text{and}\qquad f_{\mathfrak{f}_0}(h,h') = \mathfrak{f}_0 hh'+ \mathfrak{f}_0 p^{r-\eta} k \binom{h}{2} h' + \mathfrak{f}_0 b \binom{h}{2} h'.
$$
Finally, when $a=1$,
$$
h\blackdiamond y = y \qquad\text{and}\qquad f_{\mathfrak{f}_0}(h,h') = \mathfrak{f}_0 hh' +\mathfrak{f}_0 b \binom{h}{2} h'.
$$
We want to point out that in all formulas for $f_{\mathfrak{f}_0}(h,h')$ we assume that $0\le h,h'<p^{\eta}$.
\end{remark}

\begin{remark} As we saw in~\eqref{Phi barra1'} and~\eqref{Phi barra'}, in all the cases considered in Theorem~\ref{prop 4.10'}
$$
\Ho^2_{\blackdiamond,\Yleft}(H,I)\simeq \frac{\ker T_1\cap \ker T_2}{\ima S}\simeq \frac{\ker \wt{T}_1\cap \ker \wt{T}_2}{\ima \wt{S}},
$$
where $T_1$, $T_2$ and $S$ are as in~\eqref{T_1, T_2 y S}, and $\wt{T}_1$, $\wt{T}_2$ and $\wt{S}$ are as in~\eqref{wt T1, wt T2 y wt T3 caso a} or~\eqref{wt T1, wt T2 y wt T3 caso b}. In many cases the Abelian group structure of $\Ho^2_{\blackdiamond,\Yleft}(H,I)$ follows easily from the proof of Theorem~\ref{prop 4.10'}. Concretely:

\begin{itemize}

\item[-] If $p = 2$, $\eta=1$ and $r=2$, then, by~\eqref{como p=2 eta=1 r=2}, we have $(b,a)\in\{(0,1),(0,3),(2,1),(2,3)\}$ and
$$
\Ho^2_{\blackdiamond,\Yleft}(H,I) \simeq \begin{cases} \mathbb{Z}_2\oplus \mathbb{Z}_2 &\text{if $a=1$ and $b=0$,}\\ \mathbb{Z}_2 &\text{if $a=1$ and $b=2$,}\\ \mathbb{Z}_2\oplus \mathbb{Z}_2 &\text{if $a=3$ and $b=0$,}\\ \mathbb{Z}_4 &\text{if $a=3$ and $b=2$.}\end{cases}
$$

\item[-] If $p$ odd, $\eta<r$, $\eta\le q$ and $q'\le\beta$, then $\eta+q'\ge r$ and by~\eqref{para homologia 1},
$$
\Ho^2_{\blackdiamond,\Yleft}(H,I) \simeq \mathbb{Z}_{p^{\eta-r+q'}} \oplus \mathbb{Z}_{p^{\eta}}.
$$

\item[-] If $p$ odd, $\eta<r$, $q<\eta$ and $q=q'= \beta$, then, by~\eqref{para homologia 2}, we have
$$
\Ho^2_{\blackdiamond,\Yleft}(H,I) \simeq \mathbb{Z}_{p^{\beta}}\oplus \mathbb{Z}_{p^{\beta+\eta-r}}.
$$

\item[-] If $p$ odd, $\eta<r$, $q<\eta$ and $q=q'<\beta$, then $a\ne 1$, $r-q= \eta-u$ and by~\eqref{para homologia 3}, we have
$$
\Ho^2_{\blackdiamond,\Yleft}(H,I) \simeq \mathbb{Z}_{p^{r-\eta+u}}\oplus \mathbb{Z}_{p^u}.
$$
\end{itemize}
The group structure of $\Ho^2_{\blackdiamond,\Yleft}(H,I)$ was not determined in the proof of Theorem~\ref{prop 4.10'} in the following cases:

\begin{itemize}

\item[-] $p$ odd, $\eta<r$, $\eta\le q$ and $q'>\beta$,

\item[-] $p$ odd, $\eta<r$, $q<\eta$ and $q'=\beta<q$,

\item[-] $p$ odd, $\eta<r$, $q<\eta$, $q=\beta<q'$ and $q'\le r-\eta+\beta$,

\item[-] $p$ odd, $\eta<r$, $q<\eta$, $q=\beta<q'$ and $r-\eta+\beta<q'$.

\end{itemize}
\end{remark}

Next, we determine the group structure of $\Ho^2_{\blackdiamond,\Yleft}(H,I)$ in the cases that had not been resolved previously. In the proof, we make use of the $p$-adic valuation of elements of $\mathbb{Z}_{p^r}$, defined by
$$
v_p(x)\coloneqq \begin{cases} \text{the largest integer $k$ such that $p^k \mid x$} & \text{if $x \neq 0$,}\\ r & \text{if $x = 0$.} \end{cases}
$$

\begin{theorem} The following facts hold:

\begin{itemize}

\item[-] If $p$ odd, $\eta<r$, $\eta\le q$ and $q'>\beta$, then
$$
\Ho^2_{\blackdiamond,\Yleft}(H,I)\simeq \mathbb{Z}_{p^{\eta}} \oplus \mathbb{Z}_{p^{\eta+q'-r}},
$$

\item[-] If $p$ odd, $\eta<r$, $q<\eta$ and $q'=\beta<q$, then
$$
\Ho^2_{\blackdiamond,\Yleft}(H,I) \simeq \mathbb{Z}_{p^{\max(\beta,q-r+\eta)}} \oplus \mathbb{Z}_{p^{\beta+q-r+\eta-\max(\beta,q-r+\eta)}},
$$

\item[-] If $p$ odd, $\eta<r$, $q<\eta$, $q=\beta<q'$ and $q'\le r-\eta+\beta$, then
$$
\Ho^2_{\blackdiamond,\Yleft}(H,I) \simeq \begin{cases} \mathbb{Z}_{p^{\max(q'+\eta-r,q'-v_p(1-\kappa k'))}} \oplus \mathbb{Z}_{p^{q'+\eta-r+q-\max(q'+\eta-r,q'-v_p(1-\kappa k'))}} & \text{if $v_p(1-\kappa k')<q'-\beta$,}\\\mathbb{Z}_{p^{\beta}}\oplus \mathbb{Z}_{p^{q'+\eta-r}} & \text{if $v_p(1-\kappa k')\ge q'-\beta$,}\end{cases}
$$

\item[-] If $p$ odd, $\eta<r$, $q<\eta$, $q=\beta<q'$ and $r-\eta+\beta<q'$, then
$$
\Ho^2_{\blackdiamond,\Yleft}(H,I) \simeq \begin{cases} \mathbb{Z}_{p^{r-\eta+\beta-v_p(1-\kappa k')}}\oplus \mathbb{Z}_{p^{\beta-r+\eta+v_p(1-\kappa k')}} & \text{if $v_p(1-\kappa k')<r-\eta$,}\\ \mathbb{Z}_{p^{\beta}}\oplus \mathbb{Z}_{p^{\beta}} & \text{if $v_p(1-\kappa k')\ge r-\eta$.}\end{cases}
$$

\end{itemize}

\end{theorem}

\begin{proof} we consider separately each case. Throughout we will use the notations and results obtained in the proof of Theorem~\ref{prop 4.10'}. Let $(x,y)\in K$ arbitrary and let $\overline{(x,y)}$ be the class of $(x,y)$ in $K/\ima \wt{S}$. In all the cases under consideration, we obtain conditions ensuring that $\overline{(x,y)}$ attains the maximum order, and we compute $|\overline{(x,y)}|$.

\smallskip

\noindent\underline{\textsc{First case:} $p$ odd, $\eta<r$, $\eta\le q$ and $q'>\beta$}\enspace Let $x'$ and $y'$ be as in~\eqref{para despues0}, so that
$$
x = p^{r-q'}x'\quad\text{and}\quad y = -p^{r-q'+\beta-\eta}k'x' + p^{r-\eta} y'.
$$
Let $j\in \mathbb{N}$. By the last equality in~\eqref{wt(S), wt{T}_1 y wt{T}_2}, it is clear that $p^j\overline{(x,y)}=0$ if and only if
$$
p^{\eta}\mid p^jx \quad\text{and}\quad p^r\mid p^j y.
$$
The first condition is satisfied if and only if $j\ge \eta - r + q' - v_p(x')$. Since $\beta\ge \eta$, if this is the case, then the second condition is satisfied if and only if $p^jp^{r-\eta} y'=0$, which happens if and only if $j\ge \eta-v_p(y')$. Hence, the order of $\overline{(x,y)}$ is $\max\bigl(p^{\eta - r + q' - v_p(x')}, p^{\eta-v_p(y')}\bigr)$. From this it follows immediately that if $v_p(x')=v_p(y')=0$, then $\overline{(x,y)}$ attains the maximum order, and $|\overline{(x,y)}|=p^{\eta}$. Since $\left|\frac{K}{\ima\wt{S}}\right| = p^{q'+2\eta-r} > p^\eta$ (because $\eta+q'>\eta+\beta\ge r$), this implies that~$\Ho^2_{\blackdiamond,\Yleft}(H,I)$ is not a cycle group. Since $K$ is generated by two elements,
$$
\Ho^2_{\blackdiamond,\Yleft}(H,I) \simeq \frac{K}{\ima\wt{S}}\simeq \mathbb{Z}_{p^{\eta}} \oplus \mathbb{Z}_{p^{\eta+q'-r}}.
$$

\smallskip

\noindent\underline{\textsc{Second case:} $p$ odd, $\eta<r$, $q<\eta$ and $q'=\beta<q$}\enspace Let $z_1$ and $z_2$ be as in~\eqref{para despues1} and let $c$ be as in~\eqref{para despues3}, so that
$$
y=p^{r-\eta}z_2\quad\text{and}\quad x=p^{r-\beta}z_1 - cp^{r-\eta}z_2=p^{r-\beta}z_1 - p^{r-q}\kappa z_2 = p^{r-\beta}(z_1-p^{q-\beta}\kappa z_2).
$$
Let $j\in \mathbb{N}$. By the last equality in~\eqref{eq3'}, it is clear that $p^j\overline{(x,y)}=0$ if and only if
$$
p^q \mid p^jy\quad\text{and}\quad p^r \mid p^jx.
$$
The first condition holds precisely when
$$
j\ge q-r+\eta-v_p(z_2),
$$
while the second holds precisely when
$$
j\ge \beta - v_p(z_1 - p^{q-\beta}z_2).
$$
Consequently, if $v_p(z_1)=v_p(z_2)=0$, then $\overline{(x,y)}$ attains the maximum order, which is $|\overline{(x,y)}| =\max\bigl(p^{\beta},p^{q-r+\eta}\bigr)$. Since
$$
\left|\frac{K}{\ima\wt{S}}\right| = p^{\beta+q-r+\eta} >\max\bigl(p^{\beta},p^{q-r+\eta}\bigr)
$$
(because $\beta>0$ and $q-r+\eta > \beta-r+\eta\ge 0$), this implies that~$\Ho^2_{\blackdiamond,\Yleft}(H,I)$ is not a cycle group. Therefore
$$
\Ho^2_{\blackdiamond,\Yleft}(H,I) \simeq \frac{K}{\ima\wt{S}}\simeq \mathbb{Z}_{p^{\max(\beta,q-r+\eta)}} \oplus \mathbb{Z}_{p^{\beta+q-r+\eta-\max(\beta,q-r+\eta)}},
$$
because $K$ is generated by two elements.

\smallskip

\noindent\underline{\textsc{Third case:} $p$ odd, $\eta<r$, $q<\eta$, $q=\beta<q'$ and $q'\le r-\eta+\beta$}\enspace Let $x'$ and $y'$ be as in~\eqref{para despues2} and let $c$ be as in~\eqref{para despues3}, so that
$$
y = p^{r+\beta-q'-\eta}k'x' + p^{r-\eta}y' = p^{r+\beta-q'-\eta}(k'x' + p^{q'-\beta}y')\quad \text{and}\quad x = p^{r-q'}x'-cy.
$$
Since $c y = p^{r-q'}\kappa k'x' + p^{r-\beta}\kappa y'$, we have
$$
x = p^{r-q'}(1-\kappa k') x'- p^{r-\beta}\kappa y'= p^{r-q'}\bigl((1-\kappa k') x'- p^{q'-\beta}\kappa y'\bigr).
$$
Let $j\in \mathbb{N}$. By the last equality in~\eqref{eq3'}, we know that $p^j\overline{(x,y)}=0$ if and only if
$$
p^q \mid p^jy\quad\text{and}\quad p^r \mid p^jx.
$$
Write $1-\kappa k' = p^{v_p(1-\kappa k')}k^{\mathrm{iv}}$. We consider two cases. Namely $v_p(1-\kappa k')<q'-\beta$ and $v_p(1-\kappa k')\ge q'-\beta$.

\begin{description}[font=\normalfont\scshape, leftmargin=0cm]

\item[\underline{\textsc{Case a}: $v_p(1-\kappa k')<q'-\beta$}] Then
$$
y = p^{r+\beta-q'-\eta}(k'x' + p^{q'-\beta}y')\quad \text{and}\quad  x = p^{r-q'+v_p(1-\kappa k')}\bigl(k^{\mathrm{iv}} x'- p^{q'-\beta - v_p(1-\kappa k')}\kappa y'\bigr).
$$
From this it follows that if $v_p(x') = 0$, then $\overline{(x,y)}$ has maximum order. Suppose then, that $v_p(x') = 0$. A direct~com\-pu\-tation using that $q=\beta$ shows that
$$
p^q\mid p^j y \Leftrightarrow j\ge q'+\eta-r \qquad\text{and}\qquad p^r\mid p^j x \Leftrightarrow j \ge q'-v_p(1-\kappa k').
$$
Therefore $|\overline{(x,y)}| = \max\bigl(p^{q'+\eta-r},p^{q'-v_p(1-\kappa k')}\bigr)$. Since
$$
\left|\frac{K}{\ima\wt{S}}\right| = p^{q'+\eta-r+q} > \max\bigl(p^{q'+\eta-r},p^{q'-v_p(1-\kappa k')}\bigr)
$$
(because $q=\beta>0$ and $\eta+q = \eta+\beta>2\beta\ge r$), this implies that~$\Ho^2_{\blackdiamond,\Yleft}(H,I)$ is not a cycle group. Hence
$$
\Ho^2_{\blackdiamond,\Yleft}(H,I) \simeq \frac{K}{\ima\wt{S}}\simeq \mathbb{Z}_{p^{\max(q'+\eta-r,q'-v_p(1-\kappa k'))}} \oplus \mathbb{Z}_{p^{q'+\eta-r+q-\max(q'+\eta-r,q'-v_p(1-\kappa k'))}},
$$
because $K$ is generated by two elements.

\smallskip

\item[\underline{\textsc{Case b}: $v_p(1-\kappa k')\ge q'-\beta$}] Then
$$
y = p^{r+\beta-q'-\eta}(k'x' + p^{q'-\beta}y')\quad \text{and}\quad  x = p^{r-\beta}\bigl(p^{v_p(1-\kappa k')-q'+\beta} k^{\mathrm{iv}} x'- \kappa y'\bigr).
$$
Note that $p^r\mid p^{\beta}x$ and that $p^q\mid p^{\beta}y$ (since $q=\beta\ge\eta-r+q'$). Hence $|\overline{(x,y)}|\le p^{\beta}$. But if
$$
v_p(1-\kappa k')> q'-\beta\quad\text{and}\quad v_p(y') = 0,
$$
then $|\overline{(x,y)}| = p^{\beta}$; and the same happens when $v_p(1-\kappa k')= q'-\beta$ and $k^{\mathrm{iv}} x'- \kappa y'$. Since
$$
\left|\frac{K}{\ima\wt{S}}\right| = p^{q'+\eta-r+q} > p^{\beta}
$$
(because $\eta-r+q = \eta-r+\beta\ge 0$ and $q'>\beta$), the group $\Ho^2_{\blackdiamond,\Yleft}(H,I)$ is not cycle. Since $K$ is generated by two elements,
$$
\Ho^2_{\blackdiamond,\Yleft}(H,I) \simeq \frac{K}{\ima\wt{S}}\simeq \mathbb{Z}_{p^{\beta}}\oplus \mathbb{Z}_{p^{q'+\eta-r}}.
$$
\end{description}

\noindent\underline{\textsc{fourth case:} $p$ odd, $\eta<r$, $q<\eta$, $q=\beta<q'$ and $q'> r-\eta+\beta$}\enspace Let $t$ and $t'$ be as in~\eqref{para despues4} and let $c$ be as in~\eqref{para despues3}, so that
$$
y = k't + p^{r-\eta} t'\quad \text{and}\quad x = p^{\eta-\beta}t-cy = p^{\eta-\beta}(1-\kappa k')t - p^{r-\beta}\kappa t' = p^{\eta-\beta}\bigl((1-\kappa k')t - p^{r-\eta}\kappa t'\bigr).
$$
Let $j\in \mathbb{N}$. By the last equality in~\eqref{eq3'}, we know that $p^j\overline{(x,y)}=0$ if and only if
$$
p^q \mid p^jy\quad\text{and}\quad p^r \mid p^jx.
$$
Write $1-\kappa k' = p^{v_p(1-\kappa k')}k^{\mathrm{iv}}$. We consider two cases. Namely $v_p(1-\kappa k')<r-\eta$ and $v_p(1-\kappa k')\ge r-\beta$.

\begin{description}[font=\normalfont\scshape, leftmargin=0cm]

\item[\underline{\textsc{Case a}: $v_p(1-\kappa k')<r-\eta$}] Then
$$
y = k't + p^{r-\eta} t' \quad \text{and}\quad  x = p^{\eta-\beta+v_p(1-\kappa k')}\bigl(k^{\mathrm{iv}} t - p^{r-\eta - v_p(1-\kappa k')}\kappa t'\bigr).
$$
From this it follows that if $v_p(t) = 0$, then $\overline{(x,y)}$ has maximum order. Suppose then, that $v_p(t) = 0$. A direct~com\-pu\-tation using that $q=\beta$ shows that
$$
p^q\mid p^j y \Leftrightarrow j\ge q \qquad\text{and}\qquad p^r\mid p^j x \Leftrightarrow j \ge r-\eta+\beta-v_p(1-\kappa k').
$$
Since $q=\beta\le r-\eta+\beta-v_p(1-\kappa k')$, we have $|\overline{(x,y)}| = p^{r-\eta+\beta-v_p(1-\kappa k')}$. Since
$$
\left|\frac{K}{\ima\wt{S}}\right| = p^{2\beta} > p^{r-\eta+\beta-v_p(1-\kappa k')}
$$
(because $\beta+\eta>2\beta\ge r$), the group $\Ho^2_{\blackdiamond,\Yleft}(H,I)$ is not cycle. Since $K$ is generated by two elements,
$$
\Ho^2_{\blackdiamond,\Yleft}(H,I) \simeq \frac{K}{\ima\wt{S}}\simeq \mathbb{Z}_{p^{r-\eta+\beta-v_p(1-\kappa k')}}\oplus \mathbb{Z}_{p^{\beta-r+\eta+v_p(1-\kappa k')}}.
$$

\item[\underline{\textsc{Case b}: $v_p(1-\kappa k')\ge r-\eta$}] Then
$$
y = k't + p^{r-\eta} t' \quad \text{and}\quad  x = p^{r-\beta}\bigl(p^{v_p(1-\kappa k')-r+\eta} k^{\mathrm{iv}} t - \kappa t'\bigr).
$$
Since $p^q\mid p^{\beta}y$ and $p^r\mid p^{\beta}x$, we know that $|\overline{(x,y)}|\le p^{\beta}$. But if $v_p(1-\kappa k')>r-\eta$ and $v_p(t') = 0$, then $|\overline{(x,y)}| = p^{\beta}$; and the same happens when $v_p(1-\kappa k')=r-\eta$ and $v_p(k^{\mathrm{iv}} t - \kappa t') = 0$. Since
$$
\left|\frac{K}{\ima\wt{S}}\right| = p^{2\beta} > p^{\beta},
$$
the group $\Ho^2_{\blackdiamond,\Yleft}(H,I)$ is not cycle. Since $K$ is generated by two elements,
$$
\Ho^2_{\blackdiamond,\Yleft}(H,I) \simeq \frac{K}{\ima\wt{S}}\simeq \mathbb{Z}_{p^{\beta}}\oplus \mathbb{Z}_{p^{\beta}}.
$$
\end{description}
This finishes the proof of the theorem.
\end{proof}

\subsubsection[The underlying additive group of the linear cycle set \texorpdfstring{$E^{\blackdiamond,\Yleft}_{\gamma,\mathfrak{f}_0}$}{E}]{The underlying additive group of the linear cycle set \texorpdfstring{$\pmb{E^{\blackdiamond,\Yleft}_{\gamma,\mathfrak{f}_0}}$}{E}}

Recall thah $U\bigl(E^{\blackdiamond,\Yleft}_{\gamma,\mathfrak{f}_0}\bigr)$ denotes the underlying additive group of the linear cycle set $E^{\blackdiamond,\Yleft}_{\gamma,\mathfrak{f}_0}$.

\begin{remark} Assume that we are under the assumptions of Theorem~\ref{prop 4.10'} and that $p$ is even. A direct computation shows that
$$
U\bigl(E^{\blackdiamond,\Yleft}_{\gamma,\mathfrak{f}_0}\bigr)\simeq \begin{cases} \mathbb{Z}_4\oplus \mathbb{Z}_2 & \text{if $\gamma\in\{0,2\}$,}\\ \mathbb{Z}_8 & \text{if $\gamma\in\{1,3\}$.}\end{cases}
$$
\end{remark}

\begin{remark} Assume that we are under the assumptions of Theorem~\ref{prop 4.10'} and that $p$ is odd and $\eta\le q$. Write
$$
u_1\coloneqq \begin{cases} \text{the largest integer $k$ such that $p^k \mid z_1$} &\text{if $z_1\ne 0$,}\\ \eta+q'-r &\text{if $z_1 = 0$.}\end{cases}
$$
The exponent of $U\bigl(E^{\blackdiamond,\Yleft}_{\gamma,\mathfrak{f}_0}\bigr)$ is $p^{\eta+q'-u_1}$. Indeed if $z_1\ne 0$, then $p^{\eta+q'-u_1}$ is the order of $w_1$; while if $z_1 = 0$, then $p^{\eta+q'-u_1} = p^r$ is the orden of $1$. In both cases, for all $y\in I$ and $h\in H$, we have
$$
p^{\eta+q'-u_1}(y+w_h) = p^{\eta+q'-u_1}y + p^{\eta+q'- u_1} w_h = 0.
$$
Hence,
$$
U\bigl(E^{\blackdiamond,\Yleft}_{\gamma,\mathfrak{f}_0}\bigr)\simeq \mathbb{Z}_{p^{\eta+q'-u_1}}\oplus \mathbb{Z}_{p^{r-q'+u_1}}.
$$
\end{remark}

\begin{remark} Assume that we are under the assumptions of Theorem~\ref{prop 4.10'} and that $p$ is odd and $q<\eta$. Assume also that either $q=q'\le \beta$, or $q'=\beta<q$, or $q=\beta<q'\le r-\eta+\beta$; and write
$$
u_1\coloneqq \begin{cases} \text{the largest integer $k$ such that $p^k \mid z_1$} &\text{if $z_1\ne 0$,}\\ q' &\text{if $z_1 = 0$.}\end{cases}
$$
A direct computation shows that $p^{\eta+q'-u_1}$ is the order of $w_1\in U\bigl(E^{\blackdiamond,\Yleft}_{\gamma,\mathfrak{f}_0}\bigr)$ and clearly $p^r$ is the order of $1\in U\bigr(E^{\blackdiamond,\Yleft}_{\gamma,\mathfrak{f}_0}\bigl)$. Since also
$$
p^{\max(r,\eta+q'-u_1)}(y+w_h) = p^{\max(r,\eta+q'-u_1)}y + p^{\max(r,\eta+q'-u_1)} w_h = 0\quad\text{for all $y\in I$ and $h\in H$,}
$$
the exponent of $U\bigr(E^{\blackdiamond,\Yleft}_{\gamma,\mathfrak{f}_0}\bigr)$ is $p^{\max(r,\eta+q'-u_1)}$. Hence,
$$
U\bigr(E^{\blackdiamond,\Yleft}_{\gamma,\mathfrak{f}_0}\bigr)\simeq \begin{cases} \mathbb{Z}_{p^{\eta+q'-u_1}}\oplus \mathbb{Z}_{p^{r-q'+u_1}} & \text{if $\eta+q'-u_1\ge r$,}\\ \mathbb{Z}_{p^r}\oplus \mathbb{Z}_{p^{\eta}} & \text{if $\eta+q'-u_1<r$.}\end{cases}
$$
\end{remark}

\begin{remark} Assume that we are under the assumptions of Theorem~\ref{prop 4.10'} and that $p$ is odd and $q<\eta$. Assume also that $q=\beta< r-\eta+\beta<q'$, and write
$$
u_1\coloneqq \begin{cases} \text{the largest integer $k \in \mathbb{N}_0$ such that $p^k \mid z_1$} &\text{if $z_1\ne 0$,}\\ r+\beta-\eta &\text{if $z_1 = 0$.}\end{cases}
$$
A direct computation shows that $p^{r+\beta-u_1}$ is the order of $w_1\in U\bigl(E^{\blackdiamond,\Yleft}_{\gamma,\mathfrak{f}_0}\bigr)$ and clearly that $p^r$ is the order of $1\in U\bigr(E^{\blackdiamond,\Yleft}_{\gamma,\mathfrak{f}_0}\bigl)$. Since also
$$
p^{\max(r,r+\beta-u_1)}(y+w_h) = p^{\max(r,r+\beta-u_1)}y + p^{\max(r,r+\beta-u_1)} w_h = 0\quad\text{for all $y\in I$ and $h\in H$,}
$$
the exponent of $U\bigr(E^{\blackdiamond,\Yleft}_{\gamma,\mathfrak{f}_0}\bigr)$ is $p^{\max(r,r+\beta-u_1)}$. Hence,
$$
U\bigl(E^{\blackdiamond,\Yleft}_{\gamma,\mathfrak{f}_0}\bigr)\simeq \begin{cases} \mathbb{Z}_{p^{r+\beta-u_1}}\oplus \mathbb{Z}_{p^{\eta-\beta+u_1}} & \text{if $r+\beta-u_1\ge r$,}\\ \mathbb{Z}_{p^r}\oplus \mathbb{Z}_{p^{\eta}} & \text{if $r+\beta-u_1<r$.}\end{cases}
$$
Note finally that if $z_1\ne 0$, then $u_1<\beta$. Thus, $r+\beta-u_1<r$ if and only if $z_1=0$.
\end{remark}

\subsubsection[The socle and the center]{The socle and the center}

In this subsection, we work under the hypothesis of Theorem~\ref{prop 4.10'}. Our objective is to compute the socle and the center of $E^{\blackdiamond,\Yleft}_{\gamma,\mathfrak{f}_0}$, across the different cases that arise in that theorem.

\paragraph{Case $\pmb{p}$ odd}

\begin{proposition}\label{socalo caso 1 del segundo teorema} Assume that $p$ is odd. For each pair of parameters $(\gamma,\mathfrak{f}_0)$ as specified in Theorem~\ref{prop 4.10'}, we have
$$
\soc\bigl(E^{\blackdiamond,\Yleft}_{\gamma,\mathfrak{f}_0}\bigr) = \bigl\{y+w_h\in E^{\blackdiamond,\Yleft}_{\gamma,\mathfrak{f}_0}:p^{\eta-u}\mid h \text{ and } h\mathfrak{f}_0=-by\bigr\},
$$
where $u$ is at the beginning of Subsection~\ref{subseccion 2.2}, $0\le h<p^{\eta}$ and the equality $h\mathfrak{f}_0=-by$ is performed in $\mathbb{Z}_{p^r}$.
\end{proposition}

\begin{proof} Mimic the proof of Proposition~\ref{socalo caso 2 del primer teorema}.
\end{proof}

\begin{remark} Assume that we are under the hypothesis of Proposition~\ref{socalo caso 1 del segundo teorema}. If $\mathfrak{f}_0 = 0$, then
$$
\Soc\bigl(E^{\blackdiamond,\Yleft}_{\gamma,\mathfrak{f}_0}\bigr) = \bigl\{y+w_h\in E^{\blackdiamond,\Yleft}_{\gamma,\mathfrak{f}_0}: h\in p^{\eta-u}\mathbb{Z}_{p^{\eta}} \text{ and } y \in p^{r-\beta}\mathbb{Z}_{p^r}\}.
$$
Consequently $\bigl|\Soc\bigl(E^{\blackdiamond,\Yleft}_{\gamma,\mathfrak{f}_0}\bigr)\bigr| = p^{u+\beta}$. Assume then that $\mathfrak{f}_0\ne 0$ and write
$\mathfrak{f}_0 = p^{u_2}s_2$ with~$p\nmid s_2$. Let $\bar{k}'$ be the inverse of $k'$ in $\mathbb{Z}_{p^r}$. Then
$$
\Soc\bigl(E^{\blackdiamond,\Yleft}_{\gamma,\mathfrak{f}_0}\bigr) = \bigl\{y+w_h\in E^{\blackdiamond,\Yleft}_{\gamma,\mathfrak{f}_0} : h = p^{\xi}\tilde{h}\text{ with } \tilde{h}\in \mathbb{Z}_{p^{\eta}}\text{ and } y\in - p^{\xi+u_2-\beta}\tilde{h}s_2\bar{k}' + p^{r-\beta} \mathbb{Z}_{p^r} \bigr\},
$$
where $\xi\coloneqq \max(\eta-u,\beta-u_2)$. Consequently $\bigl|\Soc\bigl(E^{\blackdiamond,\Yleft}_{\gamma,\mathfrak{f}_0}\bigr)\bigr| = p^{\eta-\xi+\beta}$.
\end{remark}

\begin{proposition}\label{centro caso 1 del segundo teorema} Assume that $p$ is odd. For each pair of parameters $(\gamma,\mathfrak{f}_0)$ as specified in Theorem~\ref{prop 4.10'}, we have
$$
\Z\bigl(E^{\blackdiamond,\Yleft}_{\gamma,\mathfrak{f}_0}\bigr) = \bigl\{y+w_h\in E^{\blackdiamond,\Yleft}_{\gamma,\mathfrak{f}_0}: p^{\eta-u}\mid h,\text{ } h\mathfrak{f}_0=-by \text{ and } (a-1)y = by\bigr\},
$$
where $u$ is at the beginning of Subsection~\ref{subseccion 2.2}, $0\le h<p^{\eta}$ and the equalities are performed in $\mathbb{Z}_{p^r}$.
\end{proposition}

\begin{proof} Mimic the proof of Proposition~\ref{centro caso 2 del primer teorema}.
\end{proof}

\begin{remark} Let $\delta$ be the largest integer $k\le r$ such that $p^k\mid a-1-b = p^{r-\eta+u}s-p^{\beta}k'$. Observe that, if~$r-\eta+u\ne \beta$, then $\delta = \min(r-\eta+u,\beta)$, while if $r-\eta+u = \beta$, then $\delta\ge r-\eta+u$. Clearly
$$
\Z\bigl(E^{\blackdiamond,\Yleft}_{\gamma,\mathfrak{f}_0}\bigr) = \bigl\{y+w_h\in E^{\blackdiamond,\Yleft}_{\gamma,\mathfrak{f}_0}: p^{r-\delta}\mid y,\text{ } p^{\eta-u}\mid h \text{ and } h\mathfrak{f}_0=-by \bigr\}.
$$
Since $\min(\delta,\beta) = \min(r-\eta+u,\beta)$, if $\mathfrak{f}_0 = 0$, then
$$
\Z\bigl(E^{\blackdiamond,\Yleft}_{\gamma,\mathfrak{f}_0}\bigr) = \bigl\{y+w_h\in E^{\blackdiamond,\Yleft}_{\gamma,\mathfrak{f}_0}: y\in p^{r-\min(r-\eta+u,\beta)}\mathbb{Z}_{p^r} \text{ and }  h\in p^{\eta-u}\mathbb{Z}_{p^{\eta}} \bigr\}.
$$
Consequently $\bigl|\Z\bigl(E^{\blackdiamond,\Yleft}_{\gamma,\mathfrak{f}_0}\bigr)\bigr| = p^{u+\min(r-\eta+u,\beta)}$. Assume then that $\mathfrak{f}_0\ne 0$ and write $\mathfrak{f}_0 = z_2 = p^{u_2}s_2$ with~$p\nmid s_2$. Let~$\bar{k}'$ be the inverse of $k'$ in $\mathbb{Z}_{p^r}$ and let $\xi\coloneqq \max(\eta-u,\beta-u_2)$. A direct computation shows that
$$
\Z\bigl(E^{\blackdiamond,\Yleft}_{\gamma,\mathfrak{f}_0}\bigr) = \bigl\{y+w_h\in E^{\blackdiamond,\Yleft}_{\gamma,\mathfrak{f}_0}: h\in p^{\xi}\tilde{h} \text{ with } \tilde{h}\in \mathbb{Z}_{p^{\eta}},\text{ } y\in p^{r-\delta}\mathbb{Z}_{p^r}\text{ and } y\in -p^{\xi+u_2-\beta}\tilde{h}s_2\bar{k}' + p^{r-\beta} \mathbb{Z}_{p^r} \bigr\}.
$$
Assume that $\delta = \min(r-\eta+u,\beta)$. Since $p^{r-\delta}\mid y$, in this case $hf_0 = - by=0$. Thus,
$$
\Z\bigl(E^{\blackdiamond,\Yleft}_{\gamma,\mathfrak{f}_0}\bigr) = \bigl\{y+w_h\in E^{\blackdiamond,\Yleft}_{\gamma,\mathfrak{f}_0}: h\in p^{\xi'}\mathbb{Z}_{p^{\eta}} \text{ and }  y\in p^{r-\delta}\mathbb{Z}_{p^r}\bigr\},
$$
where $\xi'\coloneqq \max(\eta-u,r-u_2)$, and hence $\bigl|\Z\bigl(E^{\blackdiamond,\Yleft}_{\gamma,\mathfrak{f}_0}\bigr)\bigr| = p^{\eta-\xi'+\delta}$.
\end{remark}

\paragraph{Case {$\pmb{p=2}$}}

\begin{remark} Assume that $p=2$, which implies $(\eta,r) = (1,2)$. Since by~\eqref{formula para cdot en I times H'} and~\eqref{pep2pr}
$$
(y+w_h)\cdot (z+w_l) = a^hz + f_{\mathfrak{f}_0}(h,l) + ba^hyl + w_l,
$$
we obtain that, for each pair of parameters $(\gamma,\mathfrak{f}_0)$ as specified in Theorem~\ref{prop 4.10'},
$$
y+w_h\in \soc\bigl(E^{\blackdiamond,\Yleft}_{\gamma,\mathfrak{f}_0}\bigr) \quad\text{if and only if}\quad z = a^h z + f_{\mathfrak{f}_0}(h,l) + ba^h yl\text{ for all $l\in H$ and $z\in I$.}
$$
Taking $l=0$ and $z=1$, we conclude that if $y+w_h\in \soc\bigl(E^{\blackdiamond,\Yleft}_{\gamma,\mathfrak{f}_0}\bigr)$, then $a^h = 1$. Hence, we have
\begin{equation}\label{pepe5''}
y+w_h\in \soc\bigl(E^{\blackdiamond,\Yleft}_{\gamma,\mathfrak{f}_0}\bigr) \quad\text{if and only if}\quad a^h = 1 \text{ and } f_{\mathfrak{f}_0}(h,l) = - byl\text{ for all $l\in H$.}
\end{equation}
Note that
\begin{equation}\label{pepe5'}
a^h=1 \text{ if and only if } (a=1) \text{ or } (a=3\text{ and } h = 0).
\end{equation}
Moreover, by~\eqref{cociclo1 ej 4'} we have
\begin{equation}\label{1pepito}
f_{\mathfrak{f}_0}(h,l) = \mathfrak{f}_0 hl\qquad\text{for $0\le h,l\le 1$.}
\end{equation}
Combining this with~\eqref{pepe5''} and~\eqref{pepe5'}, we obtain that
\begin{equation}\label{2pepito}
\soc\bigl(E^{\blackdiamond,\Yleft}_{\gamma,\mathfrak{f}_0}\bigr) = \begin{cases} E^{\blackdiamond,\Yleft}_{\gamma,\mathfrak{f}_0} &\text{if $\mathfrak{f}_0 = 0$, $a=1$ and $b=0$,}\\ \{y+w_h\in E^{\blackdiamond,\Yleft}_{\gamma,\mathfrak{f}_0}:h=0\} &\text{if $(\mathfrak{f}_0,a)\in \{(2,1),(0,3),(1,3)\}$ and $b=0$,}\\ \{y+w_h\in E^{\blackdiamond,\Yleft}_{\gamma,\mathfrak{f}_0}:2\mid y\} &\text{if $\mathfrak{f}_0 = 0$, $a=1$ and $b=2$,}\\ \{y+w_h\in E^{\blackdiamond,\Yleft}_{\gamma,\mathfrak{f}_0}:2\mid y\text{ and } h=0\} &\text{if $\mathfrak{f}_0\in \{0,1,2,3\}$, $a=3$ and $b=2$.}\end{cases}
\end{equation}
Arguing as in Remark~\ref{caso 1 p=2}, we obtain that an element $y+w_h\in \soc\bigl(E^{\blackdiamond,\Yleft}_{\gamma,\mathfrak{f}_0}\bigr)$ belongs to $\Z\bigl(E^{\blackdiamond,\Yleft}_{\gamma,\mathfrak{f}_0}\bigr)$ if and only if $$
byh=0\qquad\text{and}\qquad y = a y + f_{\mathfrak{f}_0}(1,h).
$$
Using this, \eqref{1pepito} and \eqref{2pepito}, we obtain that
$$
\Z\bigl(E^{\blackdiamond,\Yleft}_{\gamma,\mathfrak{f}_0}\bigr) = \begin{cases} E^{\blackdiamond,\Yleft}_{\gamma,\mathfrak{f}_0} &\text{if $\mathfrak{f}_0 = 0$, $a=1$ and $b=0$,}\\ \{y+w_h\in E^{\blackdiamond,\Yleft}_{\gamma,\mathfrak{f}_0}:h=0\} &\text{if $\mathfrak{f}_0=2$, $a=1$ and $b=0$,}\\ \{y+w_h\in E^{\blackdiamond,\Yleft}_{\gamma,\mathfrak{f}_0}:2\mid y \text{ and }h=0\} &\text{if $(\mathfrak{f}_0,b)\in \{(0,0),(1,0),(0,2),(1,2),(2,2),(3,2)\}$ and $a=2$,}\\ \{y+w_h\in E^{\blackdiamond,\Yleft}_{\gamma,\mathfrak{f}_0}:2\mid y\} &\text{if $\mathfrak{f}_0 = 0$, $a=1$ and $b=2$.}\end{cases}
$$
\end{remark}

\end{document}